\documentclass[11pt,a4paper,oneside]{amsart}

\newcounter{commentcounter}

\usepackage{amsmath} 
\usepackage{amssymb} 
\usepackage{amsthm} 
\usepackage{stmaryrd} 
\usepackage[english]{babel} 
\usepackage[font=small,justification=centering]{caption} 
\usepackage[nodayofweek]{datetime}
\usepackage[shortlabels]{enumitem} 
\usepackage[T1]{fontenc} 
\usepackage[utf8]{inputenc} 
\usepackage{ifthen} 
\usepackage{mathabx} 
\usepackage{mathtools} 
\usepackage[dvipsnames]{xcolor} 
\usepackage[pdftex,  colorlinks=true, pagebackref=true]{hyperref}
\usepackage{MnSymbol}
\hypersetup{urlcolor=RoyalBlue, linkcolor=RoyalBlue,  citecolor=black}
\renewcommand*{\backref}[1]{}
\renewcommand*{\backrefalt}[4]
{
    \ifcase #1
        No citation in the text.
    \or
        Cited on Page #2.
    \else
        Cited on Pages #2.
    \fi
}
\usepackage{setspace} 
\usepackage{tikz-cd} 
\usepackage{xfrac} 
\usepackage[capitalize]{cleveref} 
 \usepackage[all,cmtip]{xy}

\makeatletter
\@namedef{subjclassname@1991}{Mathematical subject classification 1991}
\@namedef{subjclassname@2000}{Mathematical subject classification 2000}
\@namedef{subjclassname@2010}{Mathematical subject classification 2010}
\@namedef{subjclassname@2020}{Mathematical subject classification 2020}
\makeatother

\newtheorem{thm}{Theorem}[section]
\newtheorem{lemma}[thm]{Lemma}
\newtheorem{corollary}[thm]{Corollary}
\newtheorem{prop}[thm]{Proposition}

\newtheorem{question}[thm]{Question}

\numberwithin{equation}{thm}

\newtheorem{thmx}{Theorem}

\theoremstyle{definition}
\newtheorem{defn}[thm]{Definition}
\newtheorem{remark}[thm]{Remark}

\theoremstyle{plain}

    \newtheoremstyle{TheoremNum}
        {8.0pt plus 2.0pt minus 4.0pt}{8.0pt plus 2.0pt minus 4.0pt} 
        {\itshape} 
        {-0.15cm} 
        {\bfseries} 
        {.} 
        { }  
        {\thmname{#1}\thmnote{ \bfseries #3}}
    \theoremstyle{TheoremNum}
    \newtheorem{duplicate}{}

\newcommand*{\claimproofname}{My proof}

\DeclareMathOperator{\Aut}{\mathrm{Aut}}

\DeclareMathOperator{\im}{\mathrm{im}}

\newcommand{\TAP}{\mathsf{TAP}}

\newcommand{\calb}{{\mathcal{B}}}

\newcommand{\calg}{{\mathcal{G}}}

\newcommand{\calk}{{\mathcal{K}}}

\newcommand{\pialg}{\pi^{{\rm alg}}_1}

\newcommand{\PSL}{\mathrm{PSL}}

\newcommand{\onto}{\twoheadrightarrow}

\def\Z{\mathbb{Z}}
\def\R{\mathbb{R}}
\def\Q{\mathbb{Q}}
\def\C{\mathbb{C}}

\newcommand{\NN}{\mathbb{N}}
\newcommand{\ZZ}{\mathbb{Z}}
\newcommand{\CC}{\mathbb{C}}
\newcommand{\RR}{\mathbb{R}}
\newcommand{\QQ}{\mathbb{Q}}

\newcommand{\G}{\Gamma}

\usepackage{tikz}
\usetikzlibrary{arrows,quotes}
\tikzstyle{blackNode}=[fill=black, draw=black, shape=circle]

\title[Profinite rigidity of K\"ahler groups]{Profinite rigidity of K\"ahler groups: Riemann surfaces and subdirect products}

\author{Sam Hughes}
\address[S.~Hughes]{Max-Planck-Institut für Mathematik, Vivatsgasse 7, 53111 Bonn, Germany}
\email{sam.hughes.maths@gmail.com}\email{hughess@mpim-bonn.mpg.de}

\author{Claudio Llosa Isenrich}
\address[C.~Llosa Isenrich]{Department of Mathematics, University of Luxembourg, 6 Avenue de la Fonte, 4364 Esch-sur-Alzette, Luxembourg}
\email{claudio.llosaisenrich@uni.lu}

\author{Pierre Py}
\address[P.~Py]{Institut Fourier, Universit\'e Grenoble Alpes \& CNRS, 38000 Grenoble, France}
\email{pierre.py@univ-grenoble-alpes.fr}

\author{Matthew Stover}
\address[M.~Stover]{Department of Mathematics, CUNY Graduate Center, New York, NY 10016, USA}
\email{mstover@gc.cuny.edu}

\author{Stefano Vidussi}
\address[S.~Vidussi]{Department of Mathematics, University of California, Riverside, CA 92521, USA}
\email{svidussi@ucr.edu}

\subjclass[2020]{}

\begin{document}

\begin{abstract}
    This paper establishes strong profinite rigidity results for K\"ahler groups, showing that certain groups are determined within the class of residually finite K\"ahler groups by their profinite completion.  Examples include products of surface groups and certain groups with exotic finiteness properties studied earlier by Dimca--Papadima--Suciu and Llosa Isenrich. Consequently, there are aspherical smooth projective varieties that are determined up to homeomorphism by their algebraic fundamental group. The main tool is the following: the holomorphic fibrations of a closed K\"ahler manifold over hyperbolic $2$-orbifolds can be recovered from the profinite completion of its fundamental group. We also prove profinite invariance of the BNS invariant.
\end{abstract}

\maketitle

\setcounter{tocdepth}{1}
\tableofcontents

\section{Introduction}

For a smooth complex projective variety $X$, its topological fundamental group will be denoted by $\pi_1(X)$ and $\pialg(X)$ will denote its algebraic fundamental group. Motivated by the seminal work of Grothendieck \cite{Grothendieck-1970}, this paper considers the following questions:

\begin{question}
    Given an aspherical smooth projective variety with $\pi_1(X)$ residually finite:
    \begin{itemize}

        \item To what extent does $\pialg(X)$ determine $X$ up to homeomorphism?

        \item What about $\pi_1(X)$ up to isomorphism?

    \end{itemize}
\end{question}

Varieties for which the above questions have positive answers must be very special.  Indeed, building on work of Serre \cite{Ser-64} that makes use of the action of $\mathrm{Aut}(\C)$, it is known that there exist arbitrarily large collections of aspherical smooth projective varieties that are pairwise not homotopy equivalent but have isomorphic algebraic fundamental groups. See \cite{BauCatGru-15,MilneSuh, Sto-19,Sto-22} for a plethora of examples. On the other hand, the main results of this paper show that a strong form of rigidity holds for direct products of hyperbolic Riemann surfaces.

The techniques in this paper apply not only to smooth projective varieties but also to compact K\"ahler manifolds, so our results are stated in this more general context. Moreover, since the algebraic fundamental group of a smooth complex projective variety is isomorphic to the profinite completion of its topological fundamental group, from here forward this paper will use the language of profinite completions instead of algebraic fundamental groups. The notation $\widehat{G}$ will denote the profinite completion of a group $G$.

Recall that a \emph{K\"ahler group} is a group that can be realised as the fundamental group of a compact K\"ahler manifold; see~\cite{abckt} and \cite{py-book} for the essential properties of this class of groups. The first main result of this paper is profinite rigidity of products of hyperbolic surface groups amongst residually finite K\"ahler groups.

\begin{thmx}\label{cor:strong-rigidity-of-direct-product}
    Let $X$ be a compact K\"ahler manifold with residually finite fundamental group.  Assume that
    \[
        \widehat{\pi_{1}(X)} \cong \prod_{i=1}^r\widehat{\pi_{1}(S_i)}
    \]
    for some $r \ge 1$, where each $S_i$ is a closed oriented surface of genus at least $2$. Then $\pi_1(X)$ is isomorphic to $\prod_{i=1}^r \pi_1(S_i)$. Moreover, if $X$ is aspherical then it is biholomorphic to $\prod_{i=1}^r S_i$ for a suitable choice of complex structure on each of the $S_{i}$.  
\end{thmx}

This theorem relies on a much more general result, \Cref{thm:profinite-univ-hom}, which is one of our main results and which will be explained and stated below, once we introduced all relevant notions. Note that Toledo constructed examples of projective varieties with fundamental group that is not residually finite \cite{Tol-93}, so the hypothesis that $\pi_1(X)$ is residually finite is non-empty. It remains a major open problem to determine whether or not every infinite K\"ahler group admits a nontrivial finite quotient, but all known examples of infinite K\"ahler groups admit linear representations over $\mathbb{C}$ with infinite image and hence they have nontrivial profinite completion. Note however that some of our results still hold without the residual finiteness assumption; e.g., this is the case of \Cref{thm:profinite-univ-hom} below.

Interestingly, the statement of \Cref{cor:strong-rigidity-of-direct-product} is very sensitive to passing to finite index subgroups. Amongst the aforementioned examples building on those of Serre, Bauer--Catanese--Grunewald constructed varieties $X_1$, $X_2$, both quotients of a fixed product of curves by distinct free actions of a certain finite group, such that $X_1$ and $X_2$ are not homotopy equivalent but $\widehat{\pi_1(X_1)} \cong \widehat{\pi_1(X_2)}$ \cite[Thm.\ 7.9]{BauCatGru-15}. Quotients of this kind are said to be \emph{isogenous to a product}. Still, the natural bijection between the sets of finite index subgroups of residually finite groups with the same profinite completion \cite[Prop.\ 4.4]{Rei-15} and \Cref{cor:strong-rigidity-of-direct-product} immediately imply that their construction is the only way to produce examples of closed aspherical K\"ahler manifolds (with residually finite fundamental group) that have the same algebraic fundamental group as a variety isogenous to a product. To be precise:

\begin{corollary}\label{cor:BCGrigid}
    Let $Y$ be a product of finitely many hyperbolic Riemann surfaces, $F$ be a finite group acting freely and holomorphically on $Y$, and ${X_1 = Y /F}$. Suppose that $X_2$ is an aspherical compact K\"ahler manifold with $\pi_1(X_2)$ residually finite such that $\widehat{\pi_1(X_2)} \cong \widehat{\pi_1(X_1)}$. Then there is a deformation $Y^\prime$ of the complex structure on $Y$ so that $F$ still acts holomorphically and freely on $Y^\prime$ and $X_2$ is biholomorphic to $Y^\prime / F$. In particular, $X_2$ is also isogenous to a product and has the same associated finite group.
\end{corollary}

In the case where $X_1$ is a rigid variety it would be interesting to know whether or not $X_1$ and $X_2$ must be equivalent under the action of $\mathrm{Aut}(\C)$.

\medskip

Before stating further results, some classical definitions regarding the \emph{profinite genus} of a group are required. 

\subsection*{Genus of a K\"ahler group}

Recall that the \emph{profinite genus} $\mathcal{G}(G)$ of a finitely generated residually finite group $G$ is:
\[
    \mathcal{G}(G)=\left\{H \mbox{ finitely generated\ residually finite }\mid \widehat{H}\cong \widehat{G}\right\}/\cong
\]
In other words, $\mathcal{G}(G)$ is the set of isomorphism classes of finitely generated residually finite groups with profinite completion isomorphic to $\widehat{G}$. If $\calb$ is a class of residually finite groups, then $\calg_\calb(G)$ denotes $\calg(G)\cap \calb$. A finitely generated residually finite group $G$ is then \emph{profinitely rigid} (resp.\ $\mathcal{B}$-profinitely rigid) if $\mathcal{G}(G)=\left\{G\right\}$ (resp.\ $\mathcal{G}_\calb(G)=\left\{G\right\}$), and $G$ is \emph{almost profinitely rigid} (resp.\ almost $\mathcal{B}$-profinitely rigid) if $|\mathcal{G}(G)|<\infty$ (resp.\ $|\mathcal{G}_\calb(G)|<\infty$). 

There are very few examples of groups that are profinitely rigid. A well-known open problem attributed to Remeslennikov is whether free groups and surface groups are profinitely rigid. A recent breakthrough is work of Bridson, McReynolds, Reid, and Spitler \cite{bmrs} that provides examples of profinitely rigid hyperbolic $3$-manifold and $3$-orbifold groups. More relevant to this paper, Bridson, McReynolds, Spitler, and Reid later proved that there are hyperbolic triangle groups that are profinitely rigid amongst all finitely generated and residually finite groups \cite{bmrs2}. See~\cite{bmrs,reid-icm,bourbak-remy} for history and further references on profinite rigidity.

This paper considers the question of $\mathcal{K}$-profinite rigidity, where $\calk$ is the class of residually finite K\"ahler groups, proving $\mathcal{K}$-profinite rigidity for many important examples. This includes the direct products appearing in \Cref{cor:strong-rigidity-of-direct-product} and certain K\"ahler subgroups of direct products of surface groups. The examples of Serre \cite{Ser-64}, Bauer--Catanese--Grunewald \cite{BauCatGru-15}, and Stover \cite{Sto-19,Sto-22} mentioned above show that the $\mathcal{K}$-genus of a K\"ahler group can consist of more than one isomorphism class. We do not know whether a residually finite K\"ahler group can have an infinite $\calk$-profinite genus.

\medskip

This paper also studies the following notion, which was introduced by Grothendieck \cite{Grothendieck-1970}.

\begin{defn}\label{def:grorig}
    A group $H$ is \emph{Grothendieck rigid} if any homomorphism $u : K\to H$ from a finitely generated residually finite group $K$ to $H$ that induces an isomorphism on profinite completions is necessarily an isomorphism.
\end{defn}

Lastly, in this paper a hyperbolic orbisurface group will always mean a discrete \emph{cocompact} subgroup of ${\rm PSL}_{2}(\mathbb{R})$. Note that a lattice in $\PSL_2(\R)$ is a K\"ahler group if and only if it is cocompact. The fact that cocompact lattices with torsion are K\"ahler follows from a trick due to Koll\'ar~\cite[p. 7]{abckt}. So the fact that an orbisurface group is always cocompact in this paper is no loss of generality. The rank of a finitely generated abelian group $H$ means the minimal number of generators of the quotient of $H$ by its torsion subgroup. Another main result of this paper is the following.

\begin{thmx}\label{thmx.intro.rigidGroups}
    Let $G$ be a K\"ahler group that is either: 
    \begin{itemize}
        \item a hyperbolic orbisurface group, 
        \item a direct product of a finitely generated abelian group of even rank with finitely many hyperbolic orbisurface groups, or
        \item the kernel of a factorwise surjective homomorphism from a product of hyperbolic surface groups onto $\mathbb{Z}^{2}$.
    \end{itemize}   
    Then $G$ is profinitely rigid amongst residually finite K\"ahler groups, as well as Grothendieck rigid amongst residually finite K\"ahler groups.
\end{thmx}

K\"ahler groups of the form appearing in the third point of \Cref{thmx.intro.rigidGroups} do indeed exist. The first examples were constructed by Dimca, Papadima, and Suciu \cite{DimPapSuc-09-II}, then studied further and generalised by Llosa Isenrich \cite{Llo-19,Llo-20,Llo-24}. These groups are notable due to the fact that, despite being fundamental groups of closed K\"ahler manifolds and thus finitely presented, they lack higher finiteness properties: for each $n\ge 2$, there are examples that are of type $\mathsf{F}_n$ but not of type $\mathsf{F}_{n+1}$. Recall that a group $G$ is of type $\mathsf{F}_n$ if it admits a $K(G,1)$ with finite $n$-skeleton. In light of this degeneracy, the fact that these coabelian subgroups are profinitely rigid amongst K\"ahler groups seems interesting.

Using earlier classification results due to Llosa Isenrich~\cite{Llo-24}, one can use \Cref{thmx.intro.rigidGroups} to obtain the following:


\begin{corollary}\label{thmx:three}
Let $\Gamma_{1}$, $\Gamma_{2}$, $\Gamma_{3}$ be three hyperbolic orbisurface groups and
\[
    G\leq \Gamma_{1}\times \Gamma_{2}\times \Gamma_{3}
\]
be a K\"ahler group. Then $G$ contains a finite index subgroup $H$ that is profinitely rigid amongst residually finite K\"ahler groups. 
\end{corollary}


\subsection*{Fibrations}

The main tool used to establish Theorems~\ref{cor:strong-rigidity-of-direct-product} and~\ref{thmx.intro.rigidGroups} is a profinite rigidity result for fibrings showing that holomorphic fibrations from a closed K\"ahler manifold $X$ onto hyperbolic orbisurfaces are detected by $\widehat{\pi_{1}(X)}$. Before stating this more precisely, some classical facts about fibrations of K\"ahler manifolds over Riemann surfaces are now recalled.

A closed orbi-Riemann surface (from here forward, \emph{orbisurface}) is a closed Riemann surface $S$ together with a finite set of points $P\subset S$ and a collection of natural numbers $(m_{p})_{p\in P}$ such that $m_{p}\ge 2$ for all $p\in P$. This orbifold is sometimes denoted $S^*$ or $(S,P)$, but more often the symbol $S$ is used both for the orbifold and for the underlying (unmarked) Riemann surface. Then $S$ is \emph{hyperbolic} if its orbifold Euler characteristic 
\begin{equation}\label{eq:oreuch}
    \chi^{{\rm orb}} (S,P)=\chi (S)-\sum_{p\in P}\left(1-\frac{1}{m_{p}}\right)
\end{equation}
is negative.

\begin{defn}
    A \emph{fibration} from a closed complex manifold $X$ onto a Riemann surface $S$ is a surjective holomorphic map with connected fibres.
\end{defn}

\begin{remark}
    When $S$ has positive genus, fibrations also appear under the name \emph{irrational pencils} in the literature.    
\end{remark}

If $X$ is a closed complex manifold and $f : X\to S$ a fibration onto a closed Riemann surface, one can define an orbifold structure $S_{f}^{\ast}=(S,P)$ on $S$ using the multiplicities of singular fibres of $f$ in such a way that the following properties hold:
\begin{enumerate}

    \item The map $f : X \to S_{f}^{\ast}$ is holomorphic in the orbifold sense. 

    \item The natural morphism $\pi_{1}(X) \to \pi_{1}(S)$ can be lifted to a morphism 
    \[
        f_{\ast}: \pi_{1}(X)\longrightarrow \pi_{1}^{{\rm orb}}(S_{f}^{\ast})
    \]
    with finitely generated kernel, which is called the morphism \emph{induced by $f$}.

    \item Up to isomorphism of the base, there are only finitely many fibrations $f : X\to S$ for which the induced orbifold $S_{f}^{\ast}$ is hyperbolic.

    \item When $X$ is K\"ahler, the morphisms induced by fibrations to hyperbolic orbisurfaces depend only on $\pi_{1}(X)$ and not on the specific K\"ahler manifold $X$. 
\end{enumerate}
See \Cref{sec.prelims} for the definition of the group $\pi_{1}^{{\rm orb}}(S_{f}^{\ast})$. The reader should consult~\cite{Del-08}, \cite{CorSim-08}, or~\cite[Ch.\ 2]{py-book} for the definitions of a holomorphic map to an orbisurface and of the orbifold $S_{f}^{\ast}$, along with proofs of the above results.

Consequently, if $G$ is a K\"ahler group then there is a canonical finite collection $\Gamma_1, \ldots , \Gamma_r$ of hyperbolic orbisurface groups equipped with surjective homomorphisms $\varrho_i : G\to \Gamma_i$ ($1\le i \le r$) each induced by a holomorphic fibration $f$ with $S_{f}^{\ast}$ hyperbolic. In particular, each $\varrho_i$ has finitely generated kernel.

\begin{defn}\label{def:univhom}
    The homomorphism $ (\varrho_1,\dots , \varrho_r) : G \to \Gamma_1 \times \dots \times \Gamma_r$ defined in the previous paragraph is called the \emph{universal homomorphism} for $G$.    
\end{defn}

\cite[Thm.\ 9.1]{Llo-20} introduced the universal homomorphism to rephrase results of Corlette--Simpson \cite{CorSim-08} and Delzant \cite{Del-16} in terms of a homomorphism to a direct product. \Cref{sec.prelims} will precisely describe the sense in which this homomorphism is universal. The universal homomorphism played a fundamental role in many recent advances in the study of K\"ahler groups (e.g., \cite{Llo-20,Llo-24}), and the next result shows that it is a profinite invariant amongst K\"ahler groups.

\begin{duplicate}[\Cref{thm:profinite-univ-hom}]
     Let $G$ and $H$ be K\"ahler groups,
     \begin{align*}
        \phi&=(\phi_1,\dots,\phi_r):G\longrightarrow \Gamma_1\times \dots\times \Gamma_r \\
        \psi &=(\psi_1,\dots ,\psi_s): H\longrightarrow \Sigma_1\times \dots\times \Sigma_s
     \end{align*}
     be their universal homomorphisms, and $\widehat{\phi}=(\widehat{\phi}_1,\dots,\widehat{\phi}_r), \widehat{\psi}=(\widehat{\psi}_1,\dots,\widehat{\psi}_s)$ be the induced homomorphisms on profinite completions. If there is an isomorphism $\Theta: \widehat{G}\to \widehat{H}$, then $r=s$ and, up to reordering factors, $\Gamma_i$ is isomorphic to $\Sigma_i$ for each $i\in \{1,\ldots , r\}$. Moreover, $\widehat{\psi}\circ\Theta=\widehat{\phi}$ after identifying the groups $\widehat{\Gamma}_i$ and $\widehat{\Sigma}_i$ through suitable isomorphisms of profinite orbisurface groups. In particular $\im(\widehat{\psi})\cong \im(\widehat{\phi})$.
\end{duplicate}

Note that this result has no residual finiteness hypothesis. The expression \emph{profinite orbisurface group} is used as shorthand for \emph{profinite completion of an orbisurface group}. The proof of \Cref{thm:profinite-univ-hom} will use the fact that hyperbolic orbisurface groups with the same profinite completion are isomorphic, a result due to Bridson, Conder and Reid~\cite{BriConRei-16}. \Cref{thm:profinite-univ-hom} should be contrasted with the long-standing open question of Delzant--Gromov \cite{DelzantGromov05} asking for a characterisation of the possible images of universal homomorphisms of K\"ahler groups; this question was resolved by the second author in the case of a product of three orbisurface groups \cite{Llo-24}.

\subsection*{The BNS invariant}

By work of Delzant \cite{Del-BNS}, the universal homomorphism of a K\"ahler group $G$ is closely related to its Bieri--Neumann--Strebel invariant $\Sigma^1(G)$ (or BNS-invariant for short). See~\cite{bieneustr} for the original article on $\Sigma^1(G)$, as well as~\cite[Ch.\ 11]{py-book} for an introduction. If $\Gamma$ is a finitely generated group, its BNS-invariant is a subset 
\[
    \Sigma^{1}(\Gamma)\subseteq H^{1}(\Gamma,\mathbb{R})-\{0\}
\]
which is invariant by multiplication by positive scalars. It is often viewed as a subset of the sphere
\[
    \left(H^{1}(\Gamma,\mathbb{R})-\{0\}\right)\!/\mathbb{R}_{+}^{\ast},
\]
but this paper considers it as a subset of $H^{1}(\Gamma,\mathbb{R})-\{0\}$. Briefly, $\Sigma^{1}(\Gamma)$ is defined as follows. One considers a Cayley graph ${\rm Cay}(\Gamma)$ for $\Gamma$, associated to a finite symmetric generating subset. A homomorphism $\chi\in H^{1}(\Gamma, \mathbb{R})-\{0\}$ is then said to belong to $\Sigma^{1}(\Gamma)$ if the subgraph of ${\rm Cay}(\Gamma)$ spanned by vertices where $\chi$ takes nonnegative values is connected.  

The BNS-invariant is famously related to finite generation of kernels of homomorphisms to infinite abelian groups. Now, consider a K\"ahler group $G$ with universal homomorphism
\[
    (\varrho_1, \ldots , \varrho_r) : G\longrightarrow \Gamma_1 \times \cdots \times \Gamma_r
\]
with each $\Gamma_i$ a hyperbolic orbisurface group. The morphism $\varrho_i$ induces an injection $H^{1}(\Gamma_i,\mathbb{R})\hookrightarrow H^{1}(G,\mathbb{R})$ whose image is denoted by $V_i$. Delzant proved that the BNS invariant $\Sigma^{1}(G)$ coincides with the complement of the union 
\[
    \bigcup_{i=1}^{r}V_{i}
\]
of the $V_i$. Also note that if $G$ and $H$ are groups with isomorphic profinite completions, then there is a linear isomorphism $L : H^{1}(G,\QQ) \to H^{1}(H,\QQ)$. Let $L_{\RR}$ denote the $\RR$-linear extension $H^{1}(G,\RR) \to H^{1}(H,\RR)$ of $L$. One can then ask:

\begin{question}\label{question:bns}
    Can $L$ be chosen so that $L_{\RR}(\Sigma^{1}(G))=\Sigma^{1}(H)$?
\end{question}

According to \Cref{thm:profinite-univ-hom}, in the K\"ahler case the respective complements of the BNS invariants of $G$ and $H$ in $H^1(G, \R)$ and $H^1(H, \R)$ consist of the same number of linear subspaces, and the set of dimensions of these subspaces are the same for $G$ and $H$. Note that it is possible to have a triple $V_1$, $V_2$, $V_3$ such that $V_3\cap (V_1\oplus V_2) \neq \left\{0\right\}$, as shown by the examples built in~\cite{DimPapSuc-09-II}, so it is not obvious that an isomorphism on first cohomology must respect BNS invariants. However, the answer to \Cref{question:bns} is indeed affirmative for K\"ahler groups.

\begin{duplicate}[\Cref{thmx.BNS.profinite}]
    Let $G$ and $H$ be K\"ahler groups such that $\widehat{G}\cong\widehat{H}$.  Then there is a linear isomorphism $L : H^{1}(G,\QQ) \to H^{1}(H,\QQ)$ such that
    \[
        L_{\RR}(\Sigma^1(G))=\Sigma^1(H).
    \]
\end{duplicate}

The proof of \Cref{thmx.BNS.profinite} relies on work of Hughes and Kielak~\cite{HughesKielak2022}, which itself was inspired by earlier work on $3$-manifold groups by Liu~\cite{Liu23} and Jaikin-Zapirain~\cite{JZ20}. We shall not give a detailed survey of the relevant prior work here, but simply mention two facts: 
\begin{itemize}
    \item whether or not a closed $3$-manifold $M$ fibres over the circle can be detected on the profinite completion of $\pi_1(M)$ \cite{JZ20};
    \item \Cref{thmx.BNS.profinite} is, in a sense, a K\"ahler analogue to \cite[Thm.\ 1.3]{Liu23}.
\end{itemize}
\subsection*{Outline}

\Cref{sec.prelims} recalls necessary background regarding hyperbolic orbisurface groups, subgroups of direct products, and the universal homomorphism. \Cref{sec.unihom} establishes profinite invariance of the universal homomorphism (\Cref{thm:profinite-univ-hom}). Using~\Cref{thm:profinite-univ-hom}, \Cref{sec.rigidity.subdirect,sec.rigidity.tori} prove $\calk$-profinite rigidity of products of hyperbolic orbisurface groups (possibly with a free abelian factor) and of the K\"ahler coabelian subgroups appearing in \Cref{thmx.intro.rigidGroups}. This establishes part of \Cref{cor:strong-rigidity-of-direct-product} by obtaining an isomorphism on fundamental groups and proves \Cref{thmx.intro.rigidGroups}. These sections also prove that certain K\"ahler groups commensurable with groups in these classes have finite $\calk$-genus and are Grothendieck rigid amongst residually finite K\"ahler groups. \Cref{thmx:three} is also established at the end of \Cref{sec.rigidity.tori}. \Cref{sec.homeo} concludes the proof of~\Cref{cor:strong-rigidity-of-direct-product} by producing the desired biholomorphism in the aspherical setting. \Cref{sec.bns} proves profinite invariance of the BNS invariant for K\"ahler groups (\Cref{thmx.BNS.profinite}). Finally, \Cref{sec.quests} contains a few additional remarks and questions regarding Kodaira fibrations.

\subsection*{Acknowledgements}
The authors foremost express their great appreciation for Ryan Spitler's participation in the early phases of this project. They also thank Henry Wilton for conversations related to \Cref{subsec:cepg}. Hughes is grateful for the warm hospitality he received during a visit to Karlsruhe Institute of Technology and received funding from the Royal Society (grant no. IES$\backslash$R2$\backslash$222157), from the European Research Council (grant  no. 850930), and was supported by a Humboldt Research Fellowship at Universit\"at Bonn. Llosa Isenrich was partially supported by the DFG projects 281869850 and 541703614. Py is partially supported by the project GAG ANR-24-CE40-3526-01. Stover was partially supported by Grants DMS-2203555  and DMS-2506896 from the National Science Foundation, along with award SFI-MPS-TSM-00014184 from the Simons Foundation. The first four authors thank the Heilbronn Institute for a grant that facilitated this research.  We all thank the referee for carefully reading the paper.

\section{Preliminaries}\label{sec.prelims}

This section collects results required later in the paper. Specifically, it describes results concerning ramified covers of surfaces (\Cref{subsec:orraco}), subgroups of direct products (\Cref{subsec:vsp}), and the universal homomorphism for K\"ahler groups (\Cref{subsec:univho}).

\subsection{Orbisurfaces and ramified covers}\label{subsec:orraco}

Beyond presenting standard material on orbifolds and ramified covers, this subsection also establishes notation and describes a finiteness result for certain ramified covers that is needed later in the paper. A more detailed exposition concerning orbifolds can be found e.g.\ in~\cite[\S 2]{scott}. 

\medskip

Suppose that $S$ is a closed oriented surface and $P\subset S$ is a finite set of points together with integer multiplicities $(m_{p})_{p\in P}$, where $m_p \ge 2$ for each $p$. Let $S^{\ast}$ be the orbifold associated with the data consisting of the set $P$ and the multiplicities $(m_{p})_{p\in P}$ and let $\gamma_{p}$ denote a small loop in $S-P$ enclosing $p$ once. The loop $\gamma_p$ defines a conjugacy class in the group $\pi_{1}(S-P)$, and the orbifold fundamental group of $S^{\ast}$ is the quotient
\[
    \pi_{1}^{{\rm orb}}(S^{\ast})=\pi_{1}(S-P)/\llangle \gamma_{p}^{m_{p}}=1, p\in P\rrangle.
\]
When $S^{\ast}$ is hyperbolic, i.e., when its orbifold Euler characteristic \eqref{eq:oreuch} is negative, the group $\pi_{1}^{{\rm orb}}(S^{\ast})$ can be realised as a cocompact discrete subgroup of ${\rm PSL}_{2}(\RR)$. Conversely, any cocompact discrete subgroup of ${\rm PSL}_{2}(\RR)$ can be realised as the orbifold fundamental group of a hyperbolic orbifold. When the names of the orbifold points will only clutter notation, $S_{g,\underline{n}}$ will denote the orbifold whose underlying surface has genus $g\geq 0$ and $k\geq 0$ marked points of order $n_i\geq 2$ for $\underline{n}=(n_1,\dots, n_k)\in \mathbb{N}^k$. 

\medskip

The following finiteness result for ramified covers between topological surfaces is required in \Cref{sec.rigidity.subdirect}. 

\begin{prop}\label{cor:fin-many-isom-classes-of-ramified-coverings}
    Let $S_1$ be a closed oriented surface of genus one. For every $g, d\geq 2$, there are only finitely many isomorphism classes of ramified covers $f:S_g\to S_1$ of degree at most $d$ by a closed oriented surface $S_g$ of genus $g$.
\end{prop}

Here, we say that two ramified covers $f : S_g \to S_1$ and $f^\prime : S^\prime_g\to S_1$ are isomorphic if there are  homeomorphisms $\phi : S_g \to S^\prime_g $ and $\psi : S_1\to S_1$ such that $f=\psi \circ f^\prime\circ \phi$, and $\psi$ is isotopic to the identity.

\begin{proof}
    Let $f : S_g \to S_1$ be as in the statement of the proposition, $\{p_{1}, \ldots , p_{k}\}$ be the ramification points of $f$, and $\{m_{1}, \ldots , m_{k}\}$ be the corresponding multiplicities. The Riemann-Hurwitz formula yields the equation
    \[
        2g-2 = \sum_{i=1}^k (m_i-1),
    \]
    hence $1\leq k\leq 2g-2$. If $\left\{q_1,\dots, q_\ell\right\}=f(\left\{p_1,\dots,p_k\right\})$ is the collection of branch points, then $\ell \le 2g-2$. The surface $S_g$ and the map $f$ are completely determined by the covering space
    \[
        S_g -f^{-1}(\{q_1,\dots, q_\ell\})\longrightarrow S_1 -\{q_1,\dots, q_\ell\}    
    \]
    induced by $f$. Each such covering space is associated with a conjugacy class of index $d$ subgroups of the free group $\pi_1(S_1 -\{q_1,\dots, q_\ell\})$ of rank $\ell + 1$. There are finitely many such subgroups for each $\ell$ and finitely many $\ell$, hence finitely many coverings. This proves the result.
\end{proof}

\subsection{Direct products of groups and their subgroups}\label{subsec:vsp}

This section presents a characterisation of finitely presented subgroups of direct products due to Bridson, Howie, Miller, and Short \cite{BHMS-13}. Before stating it, some notation is needed.

Consider a direct product $G=G_1\times \cdots\times G_r$ of finitely many groups. For any choice of indices $1\leq i_1<\dots<i_k\leq r$,
\[
    p_{i_1,\dots,i_k}:G\longrightarrow G_{i_1}\times \dots\times G_{i_k}
\]
denotes the appropriate projection and $G_i\leq G$ denotes the canonical embedding of $G_i$ in $G$ as the $i^{th}$ direct factor. A subgroup $H\leq G$
\begin{itemize}
    \item is \emph{subdirect} if $p_{i}(H)=G_i$ for each $1\leq i\leq r$;
    \item is \emph{full} if $(G_i\cap H)\neq \left\{1\right\}$ for all $1\leq i\leq r$; and
    \item has the \emph{virtual surjection to pairs} (\emph{VSP}) property if $p_{i,j}(H)\leq G_{i}\times G_j$ has finite index in $G_i \times G_j$ for each $1\leq i < j \leq r$.
\end{itemize}
Given an arbitrary nontrivial subgroup $H\leq G$, after possibly reordering factors, there is always an embedding of $H$ as a full subdirect product ${H\leq G^\prime_1\times \cdots \times G^\prime_s}$ in a direct product of subgroups $G^\prime_{i}\leq G_i$, $1\leq i \leq s\leq r$. To see this, if there is a single factor that intersects $H$ trivially, first project away from that factor, replace $H$ with its isomorphic image under the projection, and replace the remaining factors by the projections $G_i^\prime \coloneq p_i(H)$. Repeating this process eventually yields an embedding of $H$ as a full subdirect product. 

In \cite{BHMS-13}, Bridson, Howie, Miller, and Short carried out a detailed analysis of the properties of finitely presented residually free groups. In the course of this study, they proved the following general result about subgroups of direct products (see Theorem A in~\cite{BHMS-13}). Note that the only hypothesis on each factor is that it is finitely presented. 

\begin{thm}\label{thm:SubdirProds} 
    Let $\Gamma_1, \ldots , \Gamma_r$ be finitely presented groups and
    \[
        H\leq \Gamma_1\times \cdots \times \Gamma_r
    \]
    be a full subdirect product. If $H$ has the VSP property, then it is finitely presented.
\end{thm}

There is a converse to \Cref{thm:SubdirProds} when the groups $\Gamma_i$ are residually free. Since it is not used in this paper, the reader is referred to~\cite{BHMS-13} for a precise statement. Finally, the following result about subgroups of direct products of orbisurface groups, whose key argument follows almost verbatim from \cite[Prop.\ 6.12]{BHMS-13}, is required later. The proof is included for completeness.

\begin{prop}\label{thm:SubdirProds2}
    If $H\leq \Gamma_1\times \cdots \times \Gamma_r$ and $H^\prime\leq \Gamma^\prime_1\times \cdots \times \Gamma^\prime_s$ are isomorphic full subdirect products of hyperbolic orbisurface groups, then $r=s$ and, up to permuting factors, there are isomorphisms $\Gamma_i \cong \Gamma^\prime_i$ that induce an isomorphism $H\cong H^\prime$ of subgroups.
\end{prop}
\begin{proof}
    To see that $r=s$, first note that $H_i = H \cap \Gamma_i$ is a nontrivial normal subgroup of $H$, since $H$ is full. Since $H$ is subdirect,
    \[
        p_i(H_i)\cong H_i
    \]
    is a nontrivial normal subgroup of the hyperbolic orbisurface group $\Gamma_i$. Standard facts about Fuchsian groups then imply that $p_i(H_i)$ contains an element of infinite order, hence $H_i$ also does. The $H_i$ commute with one another, thus $H$ contains a free abelian subgroup isomorphic to $\ZZ^r$. Maximal abelian subgroups of hyperbolic orbisurface groups are cyclic, which implies that $r$ is the maximal possible rank of a free abelian subgroup of $\Gamma_1 \times \cdots \times \Gamma_r$, hence of $H$. The same argument holds for $H'$, and thus $r = s$.

    To complete the proof of the theorem, it suffices to prove that any surjective homomorphism from $H$ (resp.\ $H^\prime$) to a hyperbolic orbisurface group factors through one of the projections $H\to \Gamma_i$ (resp.\ $H^\prime\to \Gamma^\prime_i$). To see that this claim suffices to prove the theorem, fix an isomorphism $\psi : H \to H^\prime$ and consider the epimorphism $p^\prime_j \circ \psi : H \to \Gamma^\prime_j$. We obtain a factorisation of $p_i^\prime \circ \psi$ through some  $p_i : H \to \Gamma_i$. Applying the same logic to the epimorphism $p_i \circ \psi^{-1} : H^\prime \to \Gamma_i$ yields an epimorphism $\Gamma^\prime_k \to \Gamma_i$ and leads to a commutative diagram   
    \begin{equation}\label{isofspty}
        \begin{tikzcd}
            H^\prime \arrow[r]\arrow[d, "p^\prime_{k}" left] & H \arrow[d,"p_i"] \arrow[r] & H^\prime \ar[d, "p^\prime_j"] \\
            \Gamma^\prime_k \arrow[r] & \Gamma_i \arrow[r] & \Gamma^\prime_j
        \end{tikzcd}
    \end{equation}
    in which the isomorphisms in the top row compose to give the identity map on $H^\prime$. Since $H^\prime$ is full, $H^\prime_j = H^\prime \cap \Gamma^\prime_j \leq H^\prime$ is nontrivial, and it is mapped to the nontrivial subgroup $p^\prime_j(H^\prime_j) \leq \Gamma^\prime_j$ in the above diagram, which implies that $p^\prime_{k}(H^\prime_j)$ must be nontrivial. This implies that $k = j$. The composition of the morphisms in the lower row equals the identity, since the isomorphisms in the top row are inverses of each other, which implies that the morphism $\Gamma_i\to \Gamma^\prime_j$ in~\eqref{isofspty} is an isomorphism. After relabeling, this gives an isomorphism
    \[
        \Gamma_1 \times \cdots \times \Gamma_r \to \Gamma^\prime_1\times \cdots \times \Gamma^\prime_r
    \]
    inducing $\psi$ as desired.

    We now turn to the proof of the factorisation claim. We content ourselves with the case $r = 2$, the general case being similar. If $H \le \Gamma_1 \times \Gamma_2$ and $\phi : H \to \Lambda$ is a surjective homomorphism onto a hyperbolic orbisurface group, then the images of $H_i = H \cap \Gamma_i$ in $\Lambda$ are commuting normal subgroups. Again, standard facts about Fuchsian groups imply that:
    \begin{enumerate}
        \item first, $\phi(H_1)$ and $\phi (H_2)$ cannot be equal and nontrivial (otherwise this would yield an abelian normal subgroup of $\Lambda$),
        \item second, at least one $\phi(H_i)$ is trivial, say $\phi(H_1)$.
    \end{enumerate}
     Since $H_1$ is the kernel of the projection of $H$ onto $\Gamma_2$, this shows that $\phi$ factors through the projection of $H$ onto $\Gamma_2$.
\end{proof}

\subsection{The universal homomorphism}\label{subsec:univho}

Let $G$ be a K\"ahler group. \Cref{def:univhom} introduced the universal homomorphism of $G$, denoted
\[
    \varrho\coloneq(\varrho_1, \ldots , \varrho_r) : G \longrightarrow \Gamma_1 \times \cdots \times \Gamma_r.
\]
This section collects several of its key properties.

\begin{prop}\label{thm:UnivHom}
    The morphism $\varrho = (\varrho_1, \ldots , \varrho_r)$ has the following properties. 
    \begin{enumerate}
        \item Each homomorphism $\varrho_i$ is surjective with finitely generated kernel.
        \item The image $\varrho (G)\leq \Gamma_1\times \cdots \times \Gamma_r$ is full subdirect and virtually surjects onto pairs. In particular, it is finitely presented.
        \item Every epimorphism $\psi: G\to \Gamma$ onto a hyperbolic orbisurface group $\Gamma$ factors through one of the $\varrho_i$ and thus also through $\varrho$. If $\psi$ moreover has finitely generated kernel, then there is an index $i\in \{1,\ldots , r\}$ and an isomorphism $\theta : \Gamma \to \Gamma_i$ such that $\theta\circ \psi=\varrho_i$.
    \end{enumerate}
\end{prop}
\begin{proof}
    Part (1) was already mentioned in the introduction. For example, see \cite[Lem.\ 4.2]{Cat-03} or \cite[Thm.\ 3]{Del-08} for a proof. Part (2) is \cite[Lem.\ 9.2]{Llo-20}.
    
    For Part (3), realise $\Gamma$ as a cocompact discrete subgroup of ${\rm PSL}_{2}(\R)$ acting on the hyperbolic plane $\mathbb{H}^2$ and $G$ as the fundamental group of a closed K\"ahler manifold $X$. Note that if $\psi$ factors through a morphism induced by a holomorphic fibration of $X$ onto an orbisurface, the base must be hyperbolic and hence $\psi$ factors through one of the $\varrho_i$. Now suppose for a contradiction that $\psi$ does not factor through such a fibration. According to~\cite[Thm.\ 9.31]{py-book}, there exists a $\psi$-equivariant map $f : \widetilde{X}\to \mathbb{H}^{2}$ that is holomorphic for one of the two ${\rm PSL}_{2}(\mathbb{R})$-invariant complex structures on $\mathbb{H}^{2}$. Moreover, the proof of \cite[Thm.\ 9.31]{py-book} shows in this case that the image of $\psi$ cannot be discrete, which is a contradiction since $\Gamma=\psi(G)$ is discrete by assumption. Therefore $\psi$ must factor through through a morphism induced by a holomorphic fibration onto a hyperbolic orbisurface, i.e., must factor through some $\varrho_i$. For the last claim, dealing with the case where $\psi$ has finitely generated kernel, see the original proof in~\cite{Cat-08} or \cite[Thm.\ 2.23]{py-book}. 
\end{proof}

The next lemma shows that the universal homomorphism for a Kähler full subdirect product of orbisurface groups is simply the inclusion in the direct product. This observation will be used in \Cref{sec.rigidity.subdirect}. 

\begin{lemma}\label{lem:fsd-inclusion-is-univ-hom}
    Let $G< \Gamma_1\times \cdots \times \Gamma_r$ be a K\"ahler full subdirect product of hyperbolic orbisurface groups $\Gamma_i$, $1\leq i \leq r$. Then the inclusion
    \[
        \phi : G\longrightarrow \Gamma_1\times \cdots \times \Gamma_r
    \]
    is the universal homomorphism for $G$.
\end{lemma}
\begin{proof}
    Since $G<\Gamma_1\times \cdots \times \Gamma_r$ is finitely presented and full subdirect, it follows from \cite[Thm.\ 6.4]{Llo-20} (see also \cite[Thm.\ 4.6]{BridsonMiller-09} for the case of hyperbolic surface groups) that the projections $\phi_i : G\to \Gamma_i$ have finitely generated kernel. Thus the $\phi_i$ are factors of the universal homomorphism, and it remains to show that there are no other factors. Applying the exact same argument as in the proof of \cite[Prop.\ 6.12]{BHMS-13} and \Cref{thm:SubdirProds2}, we see that any surjection from $G$ onto a hyperbolic orbisurface group must factor through projection onto some $\Gamma_i$, which proves that there are indeed no other factors. 
\end{proof}

\begin{remark}
    Studying the possible images of universal homomorphisms provides insights into the nature of K\"ahler groups. Constraints on such images were established in \cite{Llo-20, Llo-24}, while nontriviality of possible images was shown in \cite{DimPapSuc-09-II, Llo-20}.
\end{remark}

\section{Profinite rigidity of the universal homomorphism}\label{sec.unihom}

This section is dedicated to the proof of \Cref{thm:profinite-univ-hom}, which asserts that the universal homomorphism of a K\"ahler group is a profinite invariant. We also emphasise that this section does not assume residual finiteness of the K\"ahler groups that appear. We collect preliminary results about profinite groups in \Cref{subsec.unihom.prelim} before proving \Cref{thm:profinite-univ-hom} over the course of Sections~\ref{subsec.unihom.proof1} and~\ref{subsec.unihom.proof2}.

\subsection{Preliminaries}\label{subsec.unihom.prelim}

In what follows, $b_{1}^{(2)}(\Gamma)$ denotes the first $\ell^2$-Betti number of a discrete group $\Gamma$. See~\cite{luckbook,pansu} for the definition of this invariant. The following theorem gathers information about the first (usual and $\ell^2$) Betti number of a finitely generated subgroup of a profinite group.

\begin{prop}\label{thm:l2-betti-fingen}
    Suppose that $\Gamma$ is a finitely presented residually finite group and let $\Lambda<\widehat{\Gamma}$ be a dense finitely generated subgroup. Then $b_1(\Lambda)\geq b_1(\Gamma)$ and $b^{(2)}_1(\Lambda)\geq b^{(2)}_1(\Gamma)$.
\end{prop}
\begin{proof}
    The inequality between usual Betti numbers is \cite[Lem.\ 4.3]{Rei-15}. The inequality between the first $\ell^{2}$-Betti numbers is \cite[Prop.\ 6.3]{Rei-15} when $\Lambda$ is finitely presented, but the proof is the same under the weaker assumption that $\Lambda$ is finitely generated after applying \cite[Thm.\ 6.5]{Rei-15} by L\"uck and Osin \cite{osiluc} instead of L\"uck's approximation theorem.
\end{proof}

The next result is proved in the last paragraph of the proof of \cite[Thm.\ 3.6]{BriConRei-16} in the case where the inclusion of $\Lambda$ induces an isomorphism $\widehat{\Lambda} \to G$. The proof is analogous in the present case and is included for completeness.

\begin{lemma}\label{lem:finite-normal-profinite-image}
    Let $G$ be a profinite group and $\Lambda < G$ be a dense subgroup. If $N\unlhd \Lambda$ is a finite normal subgroup of $\Lambda$, then it is a normal subgroup of $G$.
\end{lemma}
\begin{proof}
   Assume that $N\unlhd \Lambda$ is a finite normal subgroup. If $N$ fails to be normal in $G$, then there is $x\in G$ and $n\in N$ such that the finite set
    \[
        S=\left\{xnx^{-1}n'\mid n'\in N\right\}
    \]
    does not contain the identity $1_G$ of $G$. Since $G$ is profinite, there is a surjective homomorphism $\psi: G\to Q$ to a finite group $Q$ whose restriction to $S\cup \{1_G\}$ is injective. Since $\Lambda < G$ is dense in the profinite topology, the restriction $\psi|_\Lambda : \Lambda \to Q$ is surjective. In particular, $\psi(N)\unlhd Q$ is normal, which implies that there exists some $n' \in N$ such that
    \[
        \psi(1_G) = 1_{Q} = \psi(x)\psi(n)\psi(x)^{-1}\psi(n') = \psi(xnx^{-1}n^\prime) \in \psi(S).
    \]
    However, $\psi$ is injective on $S\cup \{1_G\}$ and $1_G \notin S$, which is a contradiction. Thus $N$ is normal in $G$.
\end{proof}

The following will be critical in the proofs of several of the primary results of this paper.

\begin{corollary}\label{prop:no-finite-normal-subgroups}
    Let $\Gamma$ be a hyperbolic orbisurface group. Then $\widehat{\Gamma}$ has no nontrivial finite normal subgroups. In particular, if $\Lambda < \widehat{\Gamma}$ is a dense subgroup, then $\Lambda$ has no nontrivial finite normal subgroups.
\end{corollary}
\begin{proof}
    The fact that $\widehat{\Gamma}$ has no nontrivial finite normal subgroups is \cite[Lem.\ 2.1]{BesGruZal-14} (see also \cite[Cor.\ 5.2]{BriConRei-16}). The second assertion then follows from~\Cref{lem:finite-normal-profinite-image}.
\end{proof}

\subsection{Homomorphisms to profinite orbisurface groups}\label{subsec.unihom.proof1}

The following notation will remain fixed throughout this section. Suppose that $G$ is a K\"ahler group and $\phi=(\phi_1,\ldots , \phi_r): G\to \Gamma_1\times \dots \times \Gamma_r$ is its universal homomorphism. Let
\[
    \widehat{\phi}=\left(\widehat{\phi}_1,\dots,\widehat{\phi}_r\right): \widehat{G}\longrightarrow \widehat{\Gamma}_1\times \dots \times \widehat{\Gamma}_r
\]
denote the induced morphism on profinite completions and $\iota_{G}: G \to \widehat{G}$ (resp.\ $\iota_{\Gamma_i}: \Gamma_i \to \widehat{\Gamma}_i$) be the canonical homomorphism from $G$ (resp.\ $\Gamma_i$) to its profinite completion. The goal of this section is to show that $\widehat{\phi}$ has a universal property with respect to homomorphisms to profinite hyperbolic orbisurface groups analogous to the universal property of $\phi$. From here onward, the expression \emph{profinite hyperbolic orbisurface group} is used as shorthand for the profinite completion of a hyperbolic orbisurface group. The following result will be essential in all that follows.

\begin{lemma}\label{lem:profinite-factorisations}
    Suppose that $\widehat{\Sigma}$ is a profinite hyperbolic orbisurface group with $b_1(\Sigma)>0$ and let $\widehat{\psi}:\widehat{G}\to \widehat{\Sigma}$ be a surjective homomorphism. Then $\widehat{\psi}\circ \iota_G$ factors through a unique $\phi_i$ and thus $\widehat{\psi}$ factors through a unique $\widehat{\phi}_i$. Moreover, if $\ker(\widehat{\psi}\circ \iota_G)$ is finitely generated, then the kernels of $\widehat{\psi}\circ \iota_G$ and $\phi_i$ coincide.
\end{lemma}
\begin{proof}
    Define $\Lambda = (\widehat{\psi}\circ\iota_G)(G)$, which is a dense, finitely generated subgroup of $\widehat{\Sigma}$. Thus $b_1(\Lambda)\geq b_1(\Sigma) > 0$ and
    \[
        b_1^{(2)}(\Lambda)\geq b_1^{(2)}(\Sigma)=-\chi(\Sigma)>0
    \]
    by \Cref{thm:l2-betti-fingen}, where $\chi(\Sigma)$ denotes the orbifold Euler characteristic. Fix a nontrivial homomorphism $\nu:\Lambda \to \Z$.

    Since $b_1^{(2)}(\Lambda)\neq 0$, \cite[Thm.\ 6.8]{Gab-02} implies that either $\ker(\nu)$ is a finite normal subgroup of $\Lambda$ or it is not finitely generated. Then $\ker(\nu)$ cannot be finite by \Cref{prop:no-finite-normal-subgroups}, so it is not finitely generated. Since $\ker(\nu)$ is not finitely generated, neither is $\ker(\nu\circ \widehat{\psi}\circ \iota_G)$.

    A theorem of Napier and Ramachandran~\cite[Thm.\ 4.3]{NapRam-01} then implies that $\nu\circ \widehat{\psi}\circ \iota_G$ factors through one of the $\G_i$ for some index $i$ (see~\cite{Del-BNS} for a more general statement), yielding a commutative diagram:
    \[
        \begin{tikzcd}
            G\arrow[r, "\widehat{\psi}\circ\iota_G"] \arrow[d, "\phi_i" left] & \Lambda\arrow[d, "\nu"] \\
            \Gamma_i \arrow[r] & \Z
        \end{tikzcd}
    \]
    Since $\ker(\phi_i)\unlhd G$ is a finitely generated normal subgroup,
    \[
        N\coloneq(\widehat{\psi}\circ\iota_G)(\ker(\phi_i))\unlhd \Lambda
    \]
    is a finitely generated normal subgroup of $\Lambda$. It follows again from \cite[Thm.\ 6.8]{Gab-02} that $N$ is finite, and then from \Cref{prop:no-finite-normal-subgroups} that $N$ is trivial. Thus $\widehat{\psi}\circ\iota_G$ factors through $\phi_i$. Functoriality of profinite completions implies that $\widehat{\psi}$ factors through $\widehat{\phi}_i$, proving the existence part of the first assertion.
    
    For uniqueness, $\widehat{\psi}\circ \iota_G$ factoring through both $\phi_i$ and $\phi_j$ produces a commutative diagram:
    \[
        \begin{tikzcd}
            G \arrow[r, "\phi_j"] \ar[d, "\phi_i" left] & \Gamma_{j} \arrow[d, "u_{j}"] \\
            \Gamma_{i}\arrow[r, "u_{i}"] &\Lambda
        \end{tikzcd}
    \]
    The group $\phi_{j}(\ker (\phi_i ))$ is a finitely generated normal subgroup of the hyperbolic orbisurface group $\Gamma_j$, hence it is trivial or has finite index in $\Gamma_j$. The image of $\phi_{j}(\ker (\phi_i ))$ under $u_{j}$ is trivial, so if it had finite index in $\Gamma_j$ then $\Lambda$ would be finite, which is a contradiction. Therefore, $\phi_{j}(\ker (\phi_i ))$ is trivial, and by symmetry $\phi_{i}(\ker (\phi_j ))$ is trivial as well. This proves that $\phi_i$ and $\phi_j$ have the same kernel, hence they are equivalent and $i=j$. The uniqueness statement for $\widehat{\psi}$ follows from the uniqueness statement for $\widehat{\psi}\circ \iota_G$ and the observation that a factorisation of $\widehat{\psi}$ through $\widehat{\phi}_j$ induces a factorisation of $\widehat{\psi}\circ \iota_G$ through $\phi_j$.

    Finally, for the ``moreover'' part observe that if $\ker(\widehat{\psi}\circ \iota_G)$ is finitely generated, then $\phi_i (\ker(\widehat{\psi}\circ \iota_G))$ is a finitely generated normal subgroup of $\Gamma_i$. Since $\Lambda$ is infinite, $\phi_i (\ker(\widehat{\psi}\circ \iota_G))$ must have infinite index in $\Gamma_i$ and thus it is trivial, implying that $\ker (\phi_i ) = \ker (\widehat{\psi}\circ \iota_G)$, and so $u_i$ is an isomorphism as desired.
\end{proof}

The next goal is to generalise \Cref{lem:profinite-factorisations} to arbitrary profinite hyperbolic orbisurface groups.

\begin{corollary}\label{cor:profinite-factorisations}
    The conclusion of \Cref{lem:profinite-factorisations} holds with $\widehat{\Sigma}$ an arbitrary profinite hyperbolic orbisurface group. 
\end{corollary}

The proof of \Cref{cor:profinite-factorisations} requires the following lemma.

\begin{lemma}\label{lem:virtfactor}
    Suppose that $G$ is a K\"ahler group, $\Lambda$ is an infinite group with no nontrivial finite normal subgroups that is not virtually abelian, and $\psi : G \to \Lambda$ is a surjective homomorphism. If $G_0$ is a finite index normal subgroup of $G$ for which the restriction of $\psi$ to $G_0$ factors through a component projection for the universal homomorphism of $G_0$, then $\psi$ factors through a component projection for the universal homomorphism of $G$. 
\end{lemma}
\begin{proof}
    Fix a closed K\"ahler manifold $X$ with fundamental group $G$ and consider the finite Galois covering $X_0\to X$ corresponding to the subgroup $G_0 \le G$. By assumption, there exists a fibration $p : X_{0}\to S_0$ onto a Riemann surface such that the restriction of $\psi$ to $G_0$ factors through the morphism $p_{\ast} : G_0 \to \pi_{1}^{{\rm orb}}(S_0)$ induced by $p$, where $S_0$ is given the orbifold structure defined by the multiple fibers of $p$ and $\ker(p_\ast)$ is finitely generated. Write $\psi|_{G_{0}}=v\circ p_{\ast}$, for some morphism $v : \pi_{1}^{{\rm orb}}(S_0)\to \Lambda$.

    Fix $g\in G$ and consider the group
    \begin{equation}\label{eq:notnormalsbg}
        g \ker(p_{\ast}) g^{-1} \le G_0,
    \end{equation}
    which is a subgroup of $G_0$ since $G_0$ is normal in $G$. It is normal in $G_0$, being the image of a normal subgroup by an (outer) automorphism, and is finitely generated because $\ker (p_{\ast})$ is. Therefore
    \[
        H\coloneq p_{\ast}(g \ker(p_{\ast}) g^{-1})
    \]
    is a finitely generated normal subgroup of $\pi_{1}^{{\rm orb}}(S_0)$, which implies that it is either trivial or has finite index. 
    
Fix a base point $x_{\ast}\in X_0$ contained in a smooth fibre of $p$. In the sequel, we use $x_{\ast}$ (resp.\ its image in $X$) as base point for the fundamental group of $X_0$ (resp.\ $X$) and we denote the deck transformation induced by an element $g\in G$ by $T_g : X_0 \to X_0$. We also fix a path $\alpha$ from $x_{\ast}$ to $T_g(x_{\ast})$, and let $F\coloneq p^{-1}(p(x_{\ast}))$. Each element from the subgroup defined in~\eqref{eq:notnormalsbg} can be represented by a loop obtained by concatenating $\alpha$, a loop in $T_g(F)$ based at $T_g(x_{\ast})$, and then the inverse of $\alpha$. We will show that $T_g(F)$ is a fibre of $p$ for all $g \in G$, which implies that $G / G_0$ acts on $S_0$ and thus there is a holomorphic map from $X$ to a quotient of $S_0$.
    
    Assume that $T_g(F)$ is not contained in a fibre of $p$. Then $p(T_g(F))$ equals $S_0$ and the map $T_g(F)\to S_0$ induces a morphism on fundamental groups whose image has finite index. The latter claim is a general fact about surjective holomorphic maps $\pi : Y\to Z$ between closed complex manifolds. Such a map induces a fibration $Y-\pi^{-1}(D)\to Z-D$ where $D$ is the set of critical values of $\pi$. This fibration induces a morphism on fundamental groups with finite index image and the natural map $\pi_{1}(Z-D)\to \pi_{1}(Z)$ is onto. This proves the claim. We now return to our main argument. Since the morphism $G_0 =\pi_{1}(X_0)\to \pi_{1}(S_{0})$ can be written as a composition
    \[
        \begin{tikzcd}
            G_{0} \ar[r, "p_{\ast}"] & \pi_{1}^{{\rm orb}}(S_{0}) \ar[r, "u"] & \pi_{1}(S_{0})
        \end{tikzcd}
    \]
    and since the above discussion shows that $(u\circ p_{\ast})(g \ker (p_{\ast})g^{-1})$ has finite index in $\pi_{1}(S_{0})$, then $H$ cannot be trivial. Thus $H$ has finite index in $\pi_{1}^{{\rm orb}}(S_0)$. The equalities
    \begin{align*}
        \{1\} &=(v\circ p_{\ast})(\ker(p_{\ast}))\\
        &= \psi(\ker(p_{\ast}))\\
        &= \psi(g \ker(p_{\ast}) g^{-1})\\ 
        &= v(H)
    \end{align*}
    then imply that $\{1\}=v(H)< \Lambda$ has finite index, which is a contradiction since $\Lambda$ is infinite. Thus $T_g(F)$ must coincide with a fibre of $p$, and so $T_g$ maps any smooth fibre of $p$ to a fibre of $p$. This implies that $G / G_0$ acts on $S_0$ and there exists a holomorphic map $X \to S$, where $S$ is the quotient of $S_0$ by the finite group $G / G_0$.
    
    Let $\pi : X \to \Sigma$ be the Stein factorisation of the map $X\to S$ (see~\cite[p.\ 18]{py-book}). To see that $\psi$ factors through the morphism 
    \begin{equation}\label{eq:miplit}
    \pi_{\ast} : \pi_{1}(X)\to \pi_{1}^{{\rm orb}}(\Sigma)
    \end{equation}
    induced by $\pi$, it is enough to prove that the fundamental group of a generic fibre of $\pi$, which generates $\ker(\pi_{\ast}$), maps trivially to $\Lambda$. To see this, consider a generic fibre $F\subset X_0$ of $p$ mapping to a generic fibre $F^\prime\subset X$ of $\pi$ under the covering map $X_0\to X$. The map $F\to F^\prime$ is a finite covering map and $\pi_{1}(F)$ maps trivially to $\Lambda$. Therefore the image of $\pi_{1}(F^\prime)$ in $\Lambda$ is a finite normal subgroup, which is then trivial by hypothesis. This proves that $\psi$ factors through the morphism~\eqref{eq:miplit}. The orbifold associated with the fibration $\pi : X \to \Sigma$ is hyperbolic because $\Lambda$ is not virtually abelian, thus $\psi$ factors through a component projection for the universal homomorphism for $G$. This completes the proof.
\end{proof}

We are now prepared to prove \Cref{cor:profinite-factorisations}. 

\begin{proof}[Proof of \Cref{cor:profinite-factorisations}]
    Consider a surjective homomorphism $\widehat{\psi} : \widehat{G}\to \widehat{\Sigma}$ to a profinite hyperbolic orbisurface group. Then, there is a finite index normal surface subgroup $\Sigma_0 \unlhd \Sigma$ with $b_1(\Sigma_0)>0$. Consider the subgroup $\widehat{\Sigma}_{0}\lhd \widehat{\Sigma}$ and its inverse image
    \[
        G_{0}\coloneq (\widehat{\psi}\circ \iota_G)^{-1}(\widehat{\Sigma}_{0}).
    \]
    The homomorphism $G_{0}\to \widehat{\Sigma}_{0}$ factors through a component of the universal homomorphism for $G_0$ by \Cref{lem:profinite-factorisations}. We apply \Cref{lem:virtfactor}, with $\Lambda=\widehat{\psi}\circ \iota_{G}(G)$. Note that $\Lambda$ embeds densely into $\widehat{\Sigma}$. This implies that it satisfies all the hypotheses from~\Cref{lem:virtfactor}, the one concerning finite normal subgroups being a consequence of~\Cref{prop:no-finite-normal-subgroups}. \Cref{lem:virtfactor} thus implies that ${(\widehat{\psi}\circ \iota_{G}) : G \to \widehat{\Sigma}}$ factors through a component $\phi_i$ of the universal homomorphism for $G$. The uniqueness of this component, as well as all other remaining assertions,  are proved exactly as in \Cref{lem:profinite-factorisations}.\end{proof}

\subsection{The universal homomorphism is a profinite invariant}\label{subsec.unihom.proof2}

The purpose of this section is to prove \Cref{thm:profinite-univ-hom}, whose statement is now recalled.

\setcounter{thmx}{2}
\begin{thmx}\label{thm:profinite-univ-hom}
          Let $G$ and $H$ be K\"ahler groups,
     \begin{align*}
        \phi&=(\phi_1,\dots,\phi_r):G\longrightarrow \Gamma_1\times \dots\times \Gamma_r \\
        \psi &=(\psi_1,\dots ,\psi_s): H\longrightarrow \Sigma_1\times \dots\times \Sigma_s
     \end{align*}
     be their universal homomorphisms, and $\widehat{\phi}=(\widehat{\phi}_1,\dots,\widehat{\phi}_r), \widehat{\psi}=(\widehat{\psi}_1,\dots,\widehat{\psi}_s)$ be the induced homomorphisms on profinite completions. If there is an isomorphism $\Theta: \widehat{G}\to \widehat{H}$, then $r=s$ and, up to reordering factors, $\Gamma_i$ is isomorphic to $\Sigma_i$ for each $i\in \{1,\ldots , r\}$. Moreover, $\widehat{\psi}\circ\Theta=\widehat{\phi}$ after identifying the groups $\widehat{\Gamma}_i$ and $\widehat{\Sigma}_i$ through suitable isomorphisms of profinite orbisurface groups. In particular $\im(\widehat{\psi})\cong \im(\widehat{\phi})$.
\end{thmx}
\begin{proof}
   If $\Lambda$ is a hyperbolic orbisurface group, then the homomorphisms $\widehat{\phi}$ and $\widehat{\psi}$ are universal with respect to surjective homomorphisms $\widehat{G}\cong\widehat{H}\to \widehat{\Lambda}$ by \Cref{cor:profinite-factorisations}. For every $i\in \left\{1,\ldots, r\right\}$, the universal property of $\widehat{\psi}$ implies that there is a unique $k_i\in \left\{1,\ldots, s\right\}$ such that there exists a factorisation $\widehat{\phi_i}\circ \Theta^{-1}=\widehat{\alpha}_{k_i}\circ \widehat{\psi}_{k_i}$ with $\widehat{\alpha}_{k_i}:\widehat{\Sigma}_{k_i}\to \widehat{\Gamma}_i$ surjective.  Similarly, the universal property of $\widehat{\phi}$ implies that there is a unique $\ell_{k_i}\in \left\{1,\ldots, r\right\}$ and an epimorphism $\widehat{\beta}_{\ell_{k_i}}:\widehat{\Gamma}_{\ell_{k_i}}\to \widehat{\Sigma}_{k_i}$ such that $\widehat{\psi}_{k_i}\circ \Theta=\widehat{\beta}_{\ell_{k_i}}\circ \widehat{\phi}_{\ell_{k_i}}$. This produces a commutative diagram 
    \[
        \begin{tikzcd}
            \widehat{G} \arrow[r, "\widehat{\phi}_{\ell_{k_i}}"] \arrow[d, "\widehat{\phi}_i" left] & \widehat{\Gamma}_{\ell_{k_i}} \ar[d, "\widehat{\alpha}_{k_i}\circ \widehat{\beta}_{\ell_{k_i}}"] \\
            \widehat{\Gamma}_i\arrow[r, "{\rm id}"] & \widehat{\Gamma}_i
        \end{tikzcd}
    \]
    of epimorphisms. Uniqueness in \Cref{cor:profinite-factorisations} implies that $\ell_{k_i}=i$. This proves that the map $i\mapsto k_i$ is injective and the map $k\mapsto l_k$ is surjective. Exchanging the roles of $G$ and $H$ proves that both maps are bijections, hence $r=s$. Thus after reordering factors we can assume that $k_i=i$ and $\ell_i=i$ ($1\le i \le r$). The arguments above then yield $\widehat{\phi}_i=\widehat{\alpha}_i \circ \widehat{\beta}_i\circ \widehat{\phi}_{i}$. Thus $\widehat{\beta}_{i}$ is injective, hence it is an isomorphism. We then have
    \[
        \widehat{\psi}\circ \Theta=(\widehat{\beta_{1}}, \ldots , \widehat{\beta}_{r})\circ \widehat{\phi}
    \]
    as desired. Finally, since $\Gamma_i$ and $\Sigma_i$ are Fuchsian groups, it follows from \cite[Thm.\ 1.1]{BriConRei-16} that $\Gamma_i\cong \Sigma_i$ for each $1\leq i \leq r$. 
\end{proof}

\section{Profinite rigidity results for K\"ahler subgroups of direct products}\label{sec.rigidity.subdirect}

This section studies the profinite genus of a K\"ahler subgroup of a direct product of hyperbolic orbisurface groups. The first result is that being a full subdirect product of orbisurface groups is a profinite invariant (\Cref{thm:SDPS-profinite-rigidity-among-KGs}). This is then used to establish profinite rigidity amongst residually finite K\"ahler groups for direct products of orbisurface groups and their co-$\Z^2$ K\"ahler subgroups.

\subsection{Preliminaries}

Here we recall general results on residually free groups and separability that are needed in the next subsections.

\begin{defn}
    A subgroup $H\leq G$ is \emph{separable} if it is the intersection of a family of finite index subgroups of $G$.
\end{defn}

It follows from the definition of the profinite topology that a subgroup ${H\leq G}$ is separable if and only if it is closed in the profinite topology. Moreover, a group is residually finite if and only if the trivial subgroup is separable. A crucial result in this paper is the following theorem of Bridson and Wilton \cite[Thm.\ B]{BriWil-08}, where the second statement follows from applying \cite[Lem.\ 4.6]{Rei-15} to their result.

\begin{thm}\label{thm:Bridson-Wilton}
    Let $G$ be a finitely generated residually free group. Then all finitely presented subgroups of $G$ are separable. In particular, the inclusion $H\hookrightarrow G$ of a finitely presented subgroup induces an isomorphism between $\widehat{H}$ and the closure $\overline{H}$ of the image of $H$ in $\widehat{G}$.
\end{thm}

The next result follows from a straightforward adaptation of an argument due to Scott \cite[Proof of Lem.\ 1.1]{Sco-78}. A simple topological proof is given for completeness.

\begin{lemma}\label{lem:virtual-separability}
    Let $G$ be a group with a finite index subgroup $H<G$ such that every finitely presented subgroup of $H$ is separable in $H$. Then every finitely presented subgroup of $G$ is separable in $G$.
\end{lemma}
\begin{proof} 
    Suppose $L < G$ is finitely presented. Then $L \cap H$ is a finite index subgroup of $L$, and thus is also finitely presented. In particular, $L \cap H$ is closed in the profinite topology on $H$ and hence is closed in the profinite topology on $G$ by definition. Then $L$ is a finite union of cosets of $L \cap H$, and therefore is also closed.
\end{proof}

The following result can be viewed as a version of Grothendieck rigidity for finitely presented groups.

\begin{lemma}\label{lem:Grothendieck-rigidity}
    Let $G$ be a group with the property that every finitely presented subgroup is separable, and suppose that $H$ is finitely presented and residually finite. If there is a homomorphism $\phi: H \to G$ that induces an isomorphism on profinite completions, then $\phi$ is an isomorphism.
\end{lemma}
\begin{proof}
    First observe that $\phi$ is injective. Indeed, first $G$ is residually finite by hypothesis since the trivial group is finitely presented. Thus there is a commutative diagram
    \[
        \begin{tikzcd}
            H \arrow[r, "\phi"] \arrow[d, "\iota_{H}" left] & G \arrow[d, "\iota_G"]  \\
            \widehat{H} \arrow[r, "\widehat{\phi}"] & \widehat{G}
        \end{tikzcd}
    \]
    with the vertical homomorphisms injective and the bottom homomorphism an isomorphism, hence $\phi$ is injective as claimed. Therefore $H\cong \phi(H)<G$ is a finitely presented subgroup, and thus $\phi(H)$ is separable in $G$ by hypothesis. It then follows from \cite[Lem.\ 2.5]{LonRei-11} that $\phi(H)=G$.
\end{proof}

Since direct products of hyperbolic orbisurface groups are virtually residually free, the following is a direct consequence of the previous results in this section.

\begin{corollary}\label{cor:Bridson-Wilton}
    Let $G$ be a direct product of finitely many hyperbolic orbisurfaces and $H<G$ be a finitely presented subgroup. Then $H$ is separable in $G$. In particular, the inclusion $H\hookrightarrow G$ induces an embedding $\widehat{H}\hookrightarrow \widehat{G}$ of profinite completions for which $\widehat{H} = \widehat{G}$ if and only if $H = G$.
\end{corollary}

\subsection{Virtual direct products}

The previous results in this section are now used to prove profinite rigidity results for K\"ahler subgroups of direct products of orbisurface groups.

\begin{thm}\label{thm:SDPS-profinite-rigidity-among-KGs}
    Let $G$ and $H$ be residually finite K\"ahler groups for which there is an isomorphism $\Theta:\widehat{G}\to \widehat{H}$. If
    \[
        G< \Gamma \coloneq \Gamma_1\times \cdots \times \Gamma_r
    \]
    is a full subdirect product of hyperbolic orbisurface groups $\Gamma_i$, $1\leq i \leq r$, then there is an embedding $H<\Gamma$ as a full subdirect product. Furthermore, there is an automorphism $\Xi$ of $\widehat{\Gamma}$ consisting of a reordering of the factors of $\widehat{\Gamma}$ followed by automorphisms of the $\widehat{\Gamma}_i$ such that $\Xi(\overline{G})=\overline{H}$, where the closures of $G$ and $H$ are taken inside $\widehat{\Gamma}$. Finally, $\overline{G}$ (resp.\ $\overline{H}$) is isomorphic to $\widehat{G}$ (resp.\ $\widehat{H}$). 

\end{thm}
\begin{proof}
    The inclusion $\phi: G\hookrightarrow \Gamma_1\times \cdots \times \Gamma_r$ is the universal homomorphism for $G$ by \Cref{lem:fsd-inclusion-is-univ-hom}. \Cref{thm:profinite-univ-hom} then implies that, after possibly reordering factors, the universal homomorphism $\psi$ for $H$ has target $\Gamma$ and there is a factor-preserving automorphism $\Xi$ of $\widehat{\Gamma}$ inducing a commutative diagram:
    \[
        \begin{tikzcd}
            \widehat{\Gamma}\arrow[r, "\Xi"] & \widehat{\Gamma}\\
            \widehat{G}\arrow[r, "\Theta"] \arrow[u, "\widehat{\phi}"] & \widehat{H}\arrow[u, "\widehat{\psi}" right]
        \end{tikzcd}
    \]
    \Cref{cor:Bridson-Wilton} implies that $\widehat{\phi}$ is injective with image $\overline{G}$. Thus $\widehat{\psi}= \Xi\circ \widehat{\phi}\circ \Theta^{-1}$ is injective and thus so is $\psi$. Moreover, the closure of $H\cong \psi(H)$ is isomorphic to $\widehat{H}$ by \Cref{cor:Bridson-Wilton} again and it coincides with $\Xi(\overline{G})$. Finally, \Cref{thm:UnivHom} implies that the image of $\psi$ is full subdirect, and this completes the proof.
\end{proof}

A direct consequence of \Cref{thm:SDPS-profinite-rigidity-among-KGs} is profinite rigidity of direct products of hyperbolic orbisurface groups among residually finite K\"ahler groups. Note that direct products of hyperbolic orbisurface groups are indeed K\"ahler groups. As already mentioned in the introduction, this follows for instance from a trick due to Koll\'ar~\cite[p.\ 7]{abckt}.

\begin{corollary}\label{cor:rigidity-of-direct-products}
    Let $G=\Gamma_1\times \cdots \times \Gamma_r$ be a direct product of hyperbolic orbisurface groups and $H$ be a residually finite K\"ahler group such that there is an isomorphism $\Theta: \widehat{G}\to \widehat{H}$. Then $H\cong \Gamma_1\times \cdots\times \Gamma_r$.
\end{corollary}
\begin{proof}
    According to \Cref{thm:SDPS-profinite-rigidity-among-KGs}, there is an embedding $\psi :H\hookrightarrow \Gamma_1\times \cdots \times \Gamma_r$ as a finitely presented full subdirect product inducing an isomorphism of profinite completions. Since K\"ahler groups are finitely presented, \Cref{cor:Bridson-Wilton} implies that $\psi$ is an isomorphism.
\end{proof}

As discussed in the introduction, general K\"ahler groups are not uniquely determined by their profinite completion. This is even false for subgroups of direct products of surface groups. Indeed, Bessa, Grunewald, and Zalesskii constructed nonisomorphic finite index subgroups $P_1, P_2 \leq \Gamma_1\times \Gamma_2$ of a direct product of two hyperbolic orbisurface groups with $\widehat{P}_1\cong \widehat{P}_2$ \cite[Ex.\ 3.14]{BesGruZal-14}. Since these groups are finite index in a K\"ahler group, they are K\"ahler. However, their examples do have finite $\mathcal{K}$-profinite genus.

\begin{corollary}\label{cor:rigidity-of-direct-prods-finite-index}
    Let $\Gamma=\Gamma_1\times \cdots \times \Gamma_r$ be a direct product of hyperbolic orbisurface groups. If $G$ is a finite index subgroup of $\Gamma$, then $\calg_\calk(G)$ is finite.
\end{corollary}
\begin{proof}
    Suppose $H\in\calg_\calk(G)$. Since $G < \Gamma$ has finite index, by projecting onto a smaller direct product and replacing the $\Gamma_i$ by suitable finite index subgroups if necessary, we can assume that $G$ is a full subdirect product. \Cref{thm:SDPS-profinite-rigidity-among-KGs} provides an embedding $H \hookrightarrow \Gamma$ and an automorphism $\Xi$ of $\widehat{\Gamma}$ such that $\overline H=\Xi (\overline G)$. Both $\widehat{G} \cong \overline{G}$ and $\widehat{H} \cong \overline{H}$ then have finite index in $\widehat{\Gamma}$, and the natural bijection between finite index subgroups of a group and those of its profinite completion gives
    \[
        |\Gamma : G| = |\widehat{\Gamma} : \widehat{G}| = |\widehat{\Gamma} : \widehat{H}| = |\Gamma : H|,
    \]
    where the last equality uses the fact that $H$ is separable in $\Gamma$ by \Cref{cor:Bridson-Wilton}. Since $\Gamma$ is finitely generated, it has only finitely many subgroups of index $|\Gamma:G|$, and therefore there are only finitely many possibilities for $H$.
\end{proof}

\subsection{Co-$\mathbb{Z}^2$ subgroups}

We now address the $\mathcal{K}$-profinite genus of K\"ahler subdirect products of hyperbolic orbisurface groups that are virtually co-$\Z^2$. First of all, we will need the following consequence of \Cref{thm:SubdirProds2}.

\begin{prop}\label{prop:iso-of-ses}
    Suppose $G,H$ are finitely presented full subdirect products of a direct product $\Gamma_1\times \cdots \times \Gamma_r$ of hyperbolic orbisurface groups such that there are surjective homomorphisms $f_i: \Gamma_i \to \ZZ^{n}$ and $g_i: \Gamma_i \to \ZZ^{n}$ with $\ker\left(\sum_{i=1}^r f_i\right) = G$ and $\ker\left(\sum_{i=1}^r g_i\right) = H$. Then there is an isomorphism ${\phi: G \to H}$ if and only if, up to permuting factors of $\Gamma_1\times \cdots \times \Gamma_r$, there are isomorphisms $\phi_i: \Gamma_i \to \Gamma_i$ and $\psi:\ZZ^{n} \to \ZZ^{n}$ such that $\psi \circ f_i = g_i \circ \phi_i$.
\end{prop}
\begin{proof}
    If the isomorphisms $\psi$ and $\phi_i$ exist and have the property stated in the proposition, then the morphism 
    \[ \prod \phi_i : \Gamma_1\times \cdots \times \Gamma_r \to \Gamma_1\times \cdots \times \Gamma_r \]
    will isomorphically map $G$ to $H$. We now deal with the ``only if'' direction. If $G$ and $H$ are isomorphic, \Cref{thm:SubdirProds2} says that (after possibly permuting factors)
    there are isomorphisms $\phi_i : \Gamma_i \to \Gamma_i$  so that the isomorphism
    \[
        \phi \coloneq\prod_{i=1}^r \phi_i: \Gamma_1\times \cdots \times \Gamma_r \longrightarrow \Gamma_1\times \cdots \times \Gamma_r
    \]
    carries the normal subgroup $G$ to the normal subgroup $H$. Therefore $\phi$ induces an isomorphism $\psi: \ZZ^n \to \ZZ^n$ on the quotients satisfying
    \[
        \psi \circ \sum_{i=1}^r f_i = \sum_{i=1}^r g_i \circ \phi.
    \]
    This implies the relation $\psi \circ f_i = g_i \circ \phi_i$ for each $i\in \{1, \ldots , r\}$.
\end{proof}

We now study homomorphisms from orbisurface groups to $\ZZ^2$. Notice that if $\Gamma$ is a closed hyperbolic orbisurface group, then there is an epimorphism $p:\Gamma \to \Sigma = \pi_1 (S)$ to the fundamental group of the (not necessarily hyperbolic) surface $S$ obtained by forgetting the orbifold structure. Moreover, any surjective morphism $f:\Gamma \to \ZZ^2$ factors through $p$, since the kernel of $p$ is normally generated by torsion. Let $\mathrm{Epi}^{d}(\Gamma)$ denote the set of epimorphisms $f:\Gamma \to \ZZ^2$ inducing a degree $d$ morphism $\Sigma \to \ZZ^2$ (i.e., so that the induced map from $S$ to a torus has degree $d$). Note that the degree is considered here as a nonnegative integer. 

\begin{prop}\label{prop:unique-degree-d-covering} \hfill
\begin{enumerate}
\item For a closed hyperbolic orbisurface group $\Gamma$ and a degree $d \in \NN$ with $d \geq 2$, let $\Aut(\Gamma)$ act on $\mathrm{Epi}^{d}(\Gamma)$ by precomposition. Then the coset space $\mathrm{Epi}^{d}(\Gamma) / \Aut(\Gamma)$ consists of a single element.
\item Let $\Gamma_i$ ($1\le i \le r$) be closed hyperbolic orbisurface groups and $d_1, \ldots , d_r$ be integers greater than $1$. Fix surjections $f_i \in \mathrm{Epi}^{d_i}(\Gamma_i)$, $1\le i \le r$, and consider the group
    \[
        G \coloneq \ker(\sum_{i=1}^r f_i) < \Gamma_1\times \cdots \times \Gamma_r.
    \]
    Then the isomorphism type of $G$ depends only on the degrees $d_i$ and not on the specific choice of the $f_i$.
\end{enumerate} 
\end{prop}
\begin{proof}
    To prove the first item of the proposition, we consider a closed hyperbolic orbisurface group $\Gamma$ and let $f,g \in \mathrm{Epi}^{d}(\Gamma)$. We consider the induced injections
    \[
        f^\ast, g^\ast: \mathrm{H}^1(\ZZ^2,\ZZ) \longrightarrow \mathrm{H}^1(\Sigma,\ZZ).
    \]
    By \cite[Cor.\ 8.2]{BerEdm-84}, there is a self-homeomorphism of the surface $S$ inducing $h \in \Aut(\Sigma)$ and $h^\ast: \mathrm{H}^1(\Sigma,\ZZ) \to \mathrm{H}^1(\Sigma,\ZZ)$ so that 
    \begin{equation}\label{eq:egmiepi}
    h^\ast \circ f^\ast = g^\ast.
    \end{equation}
    This homeomorphism of $S$ can be modified by an isotopy so that it preserves the marked points of the orbifold, and hence we can lift $h$ to an automorphism $\overline{h}$ of $\Gamma$. Since $\ZZ^2$ is abelian, $f$ and $g$ are determined by $f^\ast$ and $g^\ast$, respectively, so~\eqref{eq:egmiepi} yields the equality $f \circ \overline{h} = g$. This proves that $\mathrm{Epi}^{d}(\Gamma) / \Aut(\Gamma)$ consists of a single element.

    The second item of the proposition follows immediately from the first. 
\end{proof}



We are now ready to return to profinite rigidity results. We start with the following theorem, whose proof confines us to products of surface groups, in contrast to most of our other results that apply more generally to products of orbisurface groups.

\begin{thm}\label{prop:rigidity-co-z2}
    Suppose that
    \[
    G<\Gamma \coloneq \prod_{i=1}^r\Gamma_i
    \]
    is a K\"ahler full subdirect product of hyperbolic surface groups. Suppose that there are homomorphisms $f_i: \Gamma_i\to \ZZ^2$ that are either trivial or have finite index image, $G$ is the kernel of the morphism $f = \sum_{i=1}^r f_i: \Gamma \to \ZZ^2$, and assume that $f$ is surjective. Then:
    \begin{enumerate}
        \item for each $H\in\calg_\calk(G)$ there are homomorphisms $g_i: \Gamma_i\to \ZZ^2$ that are either trivial or have finite index image with
        \[
            |\Z^2 : \im(f_i)|=|\Z^2 : \im(g_i)|
        \]
        such that $H$ is isomorphic to the kernel of the morphism
        \[
            g = \sum_{i=1}^r g_i : \Gamma \longrightarrow \Z^2;
        \]
        \item $\calg_\calk(G)$ is finite; 
        \item if all the nontrivial $f_i$ are surjective, then $\calg_\calk(G) = \{G\}$.
    \end{enumerate}
\end{thm}
\begin{proof}
    To prove (1), consider $H\in\calg_\calk(G)$ and fix an isomorphism $\Theta: \widehat{G} \to \widehat{H}$. By \Cref{thm:SDPS-profinite-rigidity-among-KGs}, we can find an embedding $H \leq \Gamma_1\times \cdots \times \Gamma_r$ as a full subdirect product and isomorphisms $\Xi_i: \widehat{\G}_i \to \widehat{\G}_i$ such that there is a commutative diagram:
    \[
        \begin{tikzcd}
            \widehat{\Gamma}\arrow[r, "\Xi"] & \widehat{\Gamma}\\
            \widehat{G}\arrow[r, "\Theta"] \arrow[u] & \widehat{H}\arrow[u]
        \end{tikzcd}
    \]
    Here the vertical arrows are embeddings induced by the inclusions of $G$ and $H$ in $\Gamma$, and $\Xi=\prod_{i=1}^r \Xi_i$. Right exactness of the profinite completion provides a commutative diagram
    \begin{equation}\label{eqn.hat.G.RE}
        \begin{tikzcd}
            1 \arrow[r] & \widehat{G} \arrow[r]  \arrow[d,"\Theta"]& \widehat{\G}_1 \times \cdots \times \widehat{\G}_r \arrow[r,"\widehat{f}"] \arrow[d,"\Xi"]& \widehat{\ZZ}^2 \arrow[r] \arrow[d,"{\rm id}"] & 1 \\
            1 \arrow[r] & \widehat{H} \arrow[r] & \widehat{\G}_1 \times \cdots \times \widehat{\G}_r \arrow[r,"\phi"] & \widehat{\ZZ}^2 \arrow[r] & 1
        \end{tikzcd}
    \end{equation}
    where $\phi$ is defined by $\phi=\widehat{f}\circ \Xi^{-1}$.
    
    Next, we claim that the kernel of the restriction of $\phi$ to $\Gamma$ is equal to $H$. By definition, this kernel is given by 
    \begin{equation} \label{eq:myadd}
        \widehat{H} \cap (\G_1 \times \cdots \times \G_r) = \overline{H} \cap (\G_1 \times \cdots \times \G_r).
    \end{equation}
    Since $H$ is separable in $\Gamma_1\times \cdots \times \Gamma_r$ by \Cref{cor:Bridson-Wilton}, the right-hand side of \Cref{eq:myadd} is equal to $H$, thus proving our claim about the intersection of $\ker(\phi)$ with $\Gamma$. It follows that $H$ is normal in $\G_1 \times \cdots \times \G_r$, and let $A$ denote the quotient. We thus get an exact sequence
    \begin{equation}\label{eq:myadd2}
        1 \longrightarrow H \longrightarrow \Gamma \longrightarrow A \longrightarrow 1.
    \end{equation}
    Separability of $H$ in $\Gamma$ implies that $A$ is residually finite.  Comparing~\eqref{eqn.hat.G.RE} and~\eqref{eq:myadd2}, one obtains an isomorphism $\widehat{A} \cong \widehat{\ZZ}^2$. Since $A$ is finitely generated and residually finite, it is then isomorphic to $\ZZ^2$.
    
    Now there are morphisms $g_i: \Gamma_i \to \ZZ^2$ and an automorphism $\psi$ of $\widehat{\Z}^2$ such that $H = \ker(\sum_{i=1}^r g_i)$ and $\phi = \psi \circ \widehat{g}$, where $g=\sum_{i=1}^r g_i$. Combining this with~\eqref{eqn.hat.G.RE}, we see that
    \[
        \widehat{f}=\psi \circ \widehat{g}\circ \Xi,
    \]
    or equivalently that
    \begin{equation}\label{eq:trucmucho}
        \widehat{f_i}=\psi \circ \widehat{g_i}\circ \Xi_i
    \end{equation}
    for each $i\in \{1, \ldots , r\}$. In particular, $\widehat{f_i}$ and $\widehat{g_i}$ are simultaneously trivial or simultaneously have finite index images. Since $\widehat{f_i}$ has image of a given finite index or is trivial exactly when $f_i$ has the same property, we see that each $g_i$ either is trivial or has image of finite index accordingly. We thus have
    \[
        |\Z^2 : \im(f_i)|=|\widehat{\Z}^2 : \im(\widehat{f}_i)|=|\widehat{\Z}^2 : \im(\widehat{g}_i)|=|\Z^2 : \im(g_i)|
    \]
    for those indices $i$ such that $f_i$ (or equivalently $g_i$) is nontrivial. 

    We now turn to the proof of (2). Write $\Gamma_i = \pi_{1}(S_{i})$, where $S_i$ is a closed oriented surface. Applying \Cref{thm:UnivHom} and \cite[Thm 6.2]{Llo-24} to $G$ and $H$, after potentially reordering the $\Gamma_i$, there is an $s \geq 3$ so that the maps $f_i,g_i$ are induced by ramified covers of $S_i$ over a torus for $i = 1, \dots, s$ and $f_i,g_i$ are trivial for $i = s+1, \dots, r$. This implies that $G = G_0 \times \Gamma_{s+1} \times \cdots \times \Gamma_r$ and $H = H_0 \times \Gamma_{s+1} \times \cdots \times \Gamma_r$ where $G_0 = \mathrm{ker}(\sum_{i=1}^s f_i)$ and $H_0 = \mathrm{ker}(\sum_{i=1}^s g_i)$ are full subdirect products of $\Gamma_1 \times \cdots \times \Gamma_s$. We now want to argue that the degrees of the ramified coverings corresponding to $f_i$ and  $g_i$ are the same for each $i=1, \dots, s$.

    Denote by $H^{1+1}$ the image of the cup product $H^1\times H^1\to H^2$.  By~\eqref{eqn.hat.G.RE} and naturality of the cup product (see also \cite{KrWi-16}), for every natural number $n\geq 1$ there is a commutative diagram:
    \[
        \begin{tikzcd}
            H^{1+1}(\ZZ^2;\Z/n) \arrow[r, "f_i^{1+1}"] & H^{1+1}(\Gamma_i;\Z/n) \\
            H^{1+1}(\widehat \Z^2;\Z/n) \arrow[r,"\widehat{f}_i^{1+1}"] \arrow[d] \arrow[u] & H^{1+1}(\widehat \Gamma_i;\Z/n) \arrow[d] \arrow[u] \\
            H^{1+1}(\widehat \Z^2;\Z/n) \arrow[r,"\widehat{g}_i^{1+1}"] \arrow[d] & H^{1+1}(\widehat \Gamma_i;\Z/n) \arrow[d] \\
            H^{1+1}(\ZZ^2;\Z/n) \arrow[r, "g_i^{1+1}"] & H^{1+1}(\Gamma_i;\Z/n)
        \end{tikzcd}
    \]
    Here the vertical maps are isomorphisms: the middle two maps are induced by $\psi$ and $\Xi_i$ (see~\eqref{eq:trucmucho}) and the top and bottom maps come from the goodness of the groups involved. (Goodness is recalled in \Cref{subsec:cepg}.) Now, since $H^{1+1}(\Gamma_i;\Z/n)= H^2(\Gamma_i;\Z/n)$ and $H^{1+1}(\Z^2;\Z/n)= H^2(\Z^2;\Z/n)$, the degree of the ramified cover inducing $f_i$ (resp.\ $g_i$) can be recovered from the values of $n$ for which the map $f_i^{1+1}$ (resp.\ $g_i^{1+1}$) is equal to $0$. Indeed, given a prime $n$, its highest power dividing these degrees can be recovered by considering the powers of $n$ for which $f_i^{1+1}$ (or $g_i^{1+1}$) vanishes. Since the vertical maps above are isomorphisms, the maps $f_i^{1+1}$, $\widehat f_i^{1+1}$, $\widehat g_i^{1+1}$, and $g_i^{1+1}$ vanish for the same values of $n$.  Therefore the degree of the ramified covers inducing $f_i$ and $g_i$ are equal and \cref{cor:fin-many-isom-classes-of-ramified-coverings} implies that $\calg_\calk(G)$ is finite.
    
    Finally, to establish (3) we notice that if the $f_i$ are surjective or trivial, then combining the equality of the degrees of $f_i$ and $g_i$ with \Cref{prop:unique-degree-d-covering}, implies that $H_0$ and $G_0$ are isomorphic, hence so are $H$ and $G$.
\end{proof}

\begin{remark}
 The fact that the $\Gamma_i$ are surface groups (rather than more general orbisurface groups) is not used in the proof of the first item of \Cref{prop:rigidity-co-z2}, although it is used in the proof of the last two items.   
\end{remark}


\section{Once more with tori}\label{sec.rigidity.tori}

The primary goal of this section is to prove $\mathcal{K}$-profinite rigidity of direct products of a free abelian group with finitely many hyperbolic orbisurface groups (see \Cref{thm:rigid-products-tori} in \Cref{subsec:cepg}). We then combine this with the results from \Cref{sec.rigidity.subdirect} to establish \Cref{thmx:three} (\Cref{subsec:threefactors}).

\subsection{Central extensions and profinite groups}\label{subsec:cepg}

We need to understand profinite properties of central extensions of groups by $\Z^k$. Before stating the first relevant result, some notions related to \emph{goodness} are required. Let $G$ be a discrete group. As noted by Serre \cite[p.\ 15]{Serre-Gal}, for $j=0,1$ and all finite $G$-modules $M$, the restriction map $H^j(\widehat{G}, M) \to H^j(G, M)$ on group cohomology induced by $G \to \widehat{G}$ is an isomorphism. A group is then called \emph{good} if this map is an isomorphism for all $j$ and \emph{$n$-good} if the restriction map continues to be an isomorphism up to and including $j=n$. See \cite[\S 7]{Rei-15} for more on goodness.

The importance of $2$-goodness lies in the connection between $H^2(G, M)$ and central extensions of $G$ by $M$. Specifically, a group $G$ is $2$-good if and only if any group extension
\[
    1 \longrightarrow N \longrightarrow E \longrightarrow G \longrightarrow 1
\]
of $G$ by a finitely generated group $N$ induces an injection $\widehat{N} \hookrightarrow \widehat{E}$ \cite[Prop.\ 7.10]{Rei-15}. Since the profinite completion is right exact, $2$-goodness is precisely what is needed to deduce that there is always an induced exact sequence
\[
    1 \longrightarrow \widehat{N} \longrightarrow \widehat{E} \longrightarrow \widehat{G} \longrightarrow 1
\]
of profinite completions.

The following lemma is essentially due to Wilton and Zalesskii \cite[Lem.\ 8.3]{WilZal-17} in the context of orientable Seifert fibre spaces (see also \cite{Morales}). The proof given below extracts the essence of their argument.

\begin{lemma}\label{lem_WilZal_SES}
    Let $G_0$ be a residually finite group that is $2$-good, and assume that for all nontrivial $\psi \in H^2(G_0, \Z^k)$ there exists a prime power $p^m$ such that the image of $\psi$ in $H^2(G_0, (\Z/p^m)^k)$ 
    is nontrivial. If
    \begin{equation}\label{eq_WilZal1}
        1 \longrightarrow \Z^k \longrightarrow G \longrightarrow G_0 \longrightarrow 1
    \end{equation}
    is a central extension, then \eqref{eq_WilZal1} splits if and only if the induced exact sequence
    \begin{equation}\label{eq_WilZal2}
        1 \longrightarrow \widehat{\Z}^k \longrightarrow \widehat{G} \longrightarrow \widehat{G}_0 \longrightarrow 1
    \end{equation}
    of profinite completions splits.
\end{lemma}
\begin{proof}
    As described before the statement of the lemma, exactness of the sequence of profinite completions follows from $2$-goodness of $G_0$. Moreover, there is a commutative diagram
    \[
        \begin{tikzcd}
            H^2(G_0, \Z^k) \arrow[r] \arrow[d] & H^2(G_0, (\Z/p^m)^k) \\
            H^2(\widehat{G}_0, \widehat{\Z}^k) \arrow[r] & H^2(\widehat{G}_0, (\Z/p^m)^k) \arrow[u, "\cong" right]
        \end{tikzcd}
    \]
    where the horizontal maps are reduction of coefficients modulo $p^m$, the left vertical map is taking profinite completions, and the right vertical map is restriction. Furthermore, it is clear that \eqref{eq_WilZal1} splitting implies \eqref{eq_WilZal2} splits. For the remainder of the proof, assume that \eqref{eq_WilZal1} does not split, and hence its characteristic class $\psi \in H^2(G_0, \Z^k)$ is nontrivial.
    
    Choose a prime power $p^m$ so that the reduction $\psi_{p^m}$ of $\psi$ modulo $p^m$ is nontrivial. Then $2$-goodness implies that there is $\widehat{\psi}_{p^m} \in H^2(\widehat{G}_0, (\Z / p^m)^k)$ whose restriction to $G_0$ is $\psi_{p^m}$. The exact sequence
    \[
        1 \longrightarrow (\Z/p^m)^k \longrightarrow \widehat{G}_p \longrightarrow \widehat{G}_0 \longrightarrow 1
    \]
    associated with $\widehat{\psi}_{p^m}$ then does not split, but it is also the reduction modulo $p^m$ of \eqref{eq_WilZal2}, and so that sequence cannot split. This proves the lemma.
\end{proof}

The content of \cite[Lem.\ 8.3]{WilZal-17} is the case where $G_0$ is a surface group and $k = 1$. The proof of the following corollary to \Cref{lem_WilZal_SES} corrects a mistake in the direction of a cohomology exact sequence in \cite{WilZal-17} and generalises their result to $k \ge 1$ and $G_0$ a product of hyperbolic orbisurface groups.

\begin{corollary}\label{cor_WilZal_SES}
    Let $k\geq 1$ be an integer and $G_0$ be a finite product of hyperbolic orbisurface groups. Then, a central extension given by a short exact sequence
    \begin{equation}\label{eq_WilZal3}
        \begin{tikzcd}
            1 \arrow[r] & \Z^k \arrow[r] & G \arrow[r] & G_0 \arrow[r] & 1
        \end{tikzcd}
    \end{equation}
    splits if and only if the induced short exact sequence
    \begin{equation}\label{eq_WilZal4}
        \begin{tikzcd}
            1 \arrow[r] & \widehat\Z^k \arrow[r] & \widehat G \arrow[r] & \widehat G_0 \arrow[r] & 1
        \end{tikzcd}
    \end{equation}
    of profinite completions splits.
\end{corollary}
\begin{proof}
    It is well-known that hyperbolic orbisurface groups are good, and the direct product of finitely many good groups is good \cite[Prop.\ 3.4]{GJZZ}, so $G_0$ is good. In particular, $G_0$ is residually finite and $2$-good. It remains to check that every nontrivial class in $H^2(G_0, \Z^k)$ has nontrivial reduction modulo some prime power $p^m$.

    For any prime power $p^m$, the exact sequence 
    \[
        \begin{tikzcd}
            1 \arrow[r] & \Z^k \arrow[r] & \Z^k \arrow[r] & (\Z/p^m)^k \arrow[r] & 1
        \end{tikzcd}
    \] 
    coming from reduction modulo $p^m$ induces a long exact sequence
    \[
        \begin{tikzcd}
            \cdots \arrow[r] & H^1(G_0, (\Z/p^m)^k) \arrow[r] & H^2(G_0, \Z^k) \arrow[r] \arrow[d, phantom, ""{coordinate, name=Z}] & H^2(G_0, \Z^k) \arrow[dll, rounded corners, to path={ -- ([xshift=2ex]\tikztostart.east) |- (Z) [near end] -| ([xshift=-2ex]\tikztotarget.west) -- (\tikztotarget)}] \\
            & H^2(G_0, (\Z/p^m)^k) \arrow[r] & H^3(G_0, \Z^k) \arrow[r] & \cdots
        \end{tikzcd}
    \]
    in cohomology. We are only interested in the segment
    \begin{equation}\label{eq:trucpleseq}
         H^2(G_0, \Z^k) \longrightarrow H^2(G_0, \Z^k) \longrightarrow H^2(G_0, (\Z/p^m)^k). 
    \end{equation}
    of this sequence. Note that the group $H^{2}(G_{0},\Z^k)$ is a finitely generated abelian group and the first arrow in~\eqref{eq:trucpleseq} is simply the multiplication by $p^m$.
    
    Choose a nonzero class $\psi\in H^{2}(G_{0},\Z^k)$. If $\psi$ is not a torsion element then it is not divisible by $p$ for sufficiently large $p$, hence it has nontrivial image in $H^2(G_0, (\Z/p)^k)$. Now suppose that $\psi$ is torsion. Identify the torsion subgroup of $H^2(G_{0},\Z^k)$ with 
    \begin{equation}\label{eq:prodgcy}
        \prod_{t=1}^{r}\Z/q_{t}^{m_{t}}\Z
    \end{equation}
    for some prime numbers $q_t$ and integers $m_t$. Since $\psi$ is nonzero, it has a nonzero component under the identification of the torsion subgroup of $H^2(G_{0},\Z^k)$ with the group in~\eqref{eq:prodgcy}. Without loss of generality, assume it has a nonzero component in the first factor. Then $\psi$ is not divisible by $q_1^{m_{1}}$, hence it has nontrivial image in $H^2(G_{0},(\Z/q_1^{m_{1}})^k)$. It follows that every nontrivial element of $H^2(G_0, \Z^k)$ has nontrivial reduction modulo some prime power, and thus \Cref{lem_WilZal_SES} applies to give the desired conclusion.
\end{proof}

\Cref{cor_WilZal_SES} is now used to prove the main result of this section.

\begin{thm}\label{thm:rigid-products-tori}
    Let $G=\Z^{k}\times \Gamma_1\times \cdots \times \Gamma_r$, where each $\Gamma_i$ is a hyperbolic orbisurface group. If $H$ is a residually finite K\"ahler group and $\widehat{\phi} : \widehat H\to\widehat G$ is an isomorphism, then $H\cong G$.
\end{thm}
\begin{proof}
    Since $H$ is residually finite, the isomorphism $\widehat{H}\cong \widehat{G}$ induces an embedding $\iota_H: H\hookrightarrow \widehat{\Z}^k\times \widehat{\Gamma}_1\times \cdots \times \widehat{\Gamma}_r$. \Cref{thm:profinite-univ-hom} then implies that the composition of $\iota_H$ with projection away from $\widehat{\Z}^k$ coincides (up to post-composition by automorphisms of the $\widehat{\Gamma}_i$) with the inclusion $H\hookrightarrow \widehat{H}$ followed by $\widehat{\psi}$, where $\psi$ is the universal homomorphism
    \[
        \psi: H \longrightarrow \prod_{i=1}^{r} \Gamma_i
    \]
    for $H$. In particular $\widehat{\psi}$ is onto. Note that $\psi(H)$ is a finitely presented subgroup of $\prod_{i=1}^r \Gamma_i$ by \Cref{thm:UnivHom}, hence it is closed in the profinite topology on $\prod_{i=1}^r \widehat{\Gamma}_i$ by \Cref{cor:Bridson-Wilton}. If $\psi(H)$ were a proper subgroup of $\prod_{i=1}^r \Gamma_i$, then the image of $\widehat{\psi}$ would be a proper subgroup of $\prod_{i=1}^{r}\widehat{\Gamma_i}$, which is a contradiction. It follows that $\psi$ is surjective. We now consider the composition of $\iota_H$ with the projection
    \[
        p_1: \widehat{\Z}^k\times \widehat{\Gamma}_1\times \cdots \times \widehat{\Gamma}_r\longrightarrow \widehat{\Z}^{k},
    \]
    which has dense image. Since $\widehat{\Z}^k$ is torsion free, $\im(p_1\circ \iota_H)$ is isomorphic to $\Z^m$ for some $m\geq k$. In particular, we obtain an embedding 
    \begin{equation}\label{eq:eohyt}
    (p_{1}\circ \iota_H,\psi) : H\hookrightarrow \Z^m\times \Gamma_1\times \cdots \times \Gamma_r
    \end{equation}
    of $H$ as a subdirect product.
    
    The proof is now completed using a simplified version of the argument for \cite[Thm.\ 5.14]{FruMor-22}. Since the $\Gamma_i$ have trivial centre, the centre of $H$ is given by $Z(H)=\ker(\psi)=H\cap \Z^m$ where $H$ is identified with its image under the embedding~\eqref{eq:eohyt}. We examine the short exact sequence
    \begin{equation}\label{eq:againsestyu}
        1\longrightarrow Z(H) \longrightarrow H \longrightarrow \Gamma_1 \times \cdots \times \Gamma_r \longrightarrow 1, 
    \end{equation}
    where the surjection from $H$ to $\Gamma_1\times \cdots \times \Gamma_r$ is given by $\psi$. By right-exactness of taking profinite completions and $2$-goodness of the group $\Gamma_1\times \cdots \times \Gamma_r$, this yields the short exact sequence
    \begin{equation}\label{eq:againsestyuu}
        1\longrightarrow\widehat{Z(H)} \longrightarrow \widehat{H} \longrightarrow \widehat{\Gamma}_1\times \cdots \times \widehat{\Gamma}_r\longrightarrow 1.    
    \end{equation}
    Thus the image of $\widehat{Z(H)}$ in $\widehat{H}$ coincides with $\ker (\widehat{\psi})\cong \widehat{\Z}^{k}$. It follows that $Z(H)$ is isomorphic to $\Z^k$. 

    Since $\widehat{H}$ is isomorphic to $\widehat{\Z}^k\times \widehat{\Gamma}_1\times \cdots \times \widehat{\Gamma}_r$ in such a way that the projection $\widehat{H}\to \widehat{\Gamma}_1\times \cdots \times \widehat{\Gamma}_r$ coincides with $\widehat{\psi}$ up to an automorphism of the target, the sequence~\eqref{eq:againsestyuu} splits. \Cref{cor_WilZal_SES} then implies that~\eqref{eq:againsestyu} splits. In other words, $H$ is isomorphic to $\Z^k \times  \Gamma_1 \times \cdots \times \Gamma_r$.
\end{proof}

Finally, observe that \Cref{lem:Grothendieck-rigidity} and finite presentability of K\"ahler groups imply the following which, using the vocabulary from \Cref{def:grorig}, says that $G$ is Grothendieck rigid amongst residually finite K\"ahler groups.

\begin{corollary}\label{cor:grotrigi}
    Suppose that $G$ is a K\"ahler subgroup contained in the product of finitely many hyperbolic orbisurface groups with possibly a free abelian group of finite rank. If $H\in \mathcal{G}_{\mathcal{K}}(G)$ and if there is a morphism $\phi: H \to G$ that induces an isomorphism on profinite completions, then $\phi$ is an isomorphism.
\end{corollary}
\begin{proof}
    Consider the inclusion
    \[
        G< P\coloneq \Z^k \times \Gamma_1 \times \cdots \times \Gamma_r,
    \]
    where the $\Gamma_i$ are hyperbolic orbisurface groups. According to \Cref{thm:Bridson-Wilton} and \Cref{lem:virtual-separability}, every finitely presented subgroup of $P$ is separable in $P$. In particular every finitely presented subgroup of $G$ is separable in $P$, hence separable in $G$. \Cref{lem:Grothendieck-rigidity} now yields the desired conclusion. 
\end{proof}

At this point we have proved \Cref{thmx.intro.rigidGroups}. We now summarize where the various statements from that theorem are proved.

\begin{proof}[Proof of \Cref{thmx.intro.rigidGroups}]
    The first two points are consequences of \Cref{thm:rigid-products-tori} (or of \Cref{cor:rigidity-of-direct-products} for the first one), while the third follows from \Cref{prop:rigidity-co-z2}. The last claim in \Cref{thmx.intro.rigidGroups} follows from \Cref{cor:grotrigi}.
\end{proof}

\subsection{The three factors case}\label{subsec:threefactors}

Here we use the results from the previous sections to prove the following corollary, already stated in the introduction.

\begin{duplicate}[\Cref{thmx:three}]

    Let $\Gamma_{1}$, $\Gamma_{2}$, $\Gamma_{3}$ be three hyperbolic orbisurface groups and
    \[
        G\leq \Gamma_{1}\times \Gamma_{2}\times \Gamma_{3}
    \]
    be a K\"ahler group. Then $G$ contains a finite index subgroup $H$ that is profinitely rigid amongst residually finite K\"ahler groups. 
\end{duplicate}
\begin{proof}
    Let $G$ be as in the statement of the theorem. According to \cite[Cor.\ 1.3]{Llo-24}, we have the following dichotomy:
    \begin{enumerate}
        \item either $G$ has a finite index subgroup $G_{0}$ that is isomorphic to a direct product 
        \[
            \mathbb{Z}^{2k}\times \Sigma_{g_{1}}\times \cdots \times \Sigma_{g_{s}}
        \]
        where $2k+s\le 3$ and $\Sigma_{g_{i}}$ is a surface group of genus $g_i\ge 2$;

        \item or $G$ has a finite index subgroup $G_0$ that embeds as a normal full subdirect product of three hyperbolic surface groups $\Sigma_{g_{1}}\times \Sigma_{g_{2}}\times \Sigma_{g_{3}}$ with quotient isomorphic to $\Z^2$ and such that the quotient morphism $\phi :  \Sigma_{g_{1}}\times \Sigma_{g_{2}}\times \Sigma_{g_{3}} \to \mathbb{Z}^2$ is nontrivial (in fact induced by a nonconstant holomorphic map) on each factor.
    \end{enumerate}
    In Case (1) the assertion is an immediate consequence of \Cref{thm:rigid-products-tori}. In Case (2), the restriction of $\phi$ to factors is induced by a nonconstant holomorphic map and thus has finite index image $A_i\coloneq \phi(\Sigma_{g_i})\leq \mathbb{Z}^2$. Let
    \[
        A\coloneq \bigcap_{i=1}^3 A_i \cong \mathbb{Z}^2.
    \]
    Then the subgroups $\Sigma_i^\prime\coloneq  \phi^{-1}(A)\cap \Sigma_i\leq \Sigma_i$ have finite index, the restriction
    \[
        \phi^\prime : \Sigma^\prime_1\times \Sigma^\prime_2 \times \Sigma^\prime_3 \longrightarrow A\cong\mathbb{Z}^2
    \]
    of $\phi$ to $\Sigma^\prime_1\times \Sigma^\prime_2\times \Sigma^\prime_3$ is surjective on factors and $G_0^\prime\coloneq \ker(\phi^\prime)\leq G_0$ has finite index. The assertion then follows from \Cref{prop:rigidity-co-z2}.  
\end{proof}

\section{Strong rigidity for direct products of Riemann surfaces}\label{sec.homeo}

\Cref{thm:rigid-products-tori} says in particular that the profinite completion completely characterises direct products of hyperbolic surface groups (possibly with a free Abelian factor) among all residually finite K\"ahler groups. We prove here that for closed aspherical K\"ahler manifolds, more can be said. Namely, we establish the following theorem, which is more general than~\Cref{cor:strong-rigidity-of-direct-product}.  

\begin{thm}\label{thm:aspherical-case}
    Let $X$ be a closed aspherical K\"ahler manifold with residually finite fundamental group. Let $S_{1}, \ldots, S_{r}$ be closed oriented surfaces of genus greater than $1$ and let $k$ be a natural number. 
    \begin{enumerate}
        \item If the profinite completion of $\pi_{1}(X)$ is isomorphic to that of 
            \[
                \pi_{1}(S_{1})\times \cdots \times \pi_{1}(S_{r})\times \mathbb{Z}^{2k},
            \]
        then for each $j\in {1, \ldots , r}$, there exists a closed Riemann surface $F_j$ homeomorphic to $S_{j}$ such that $X$ is biholomorphic to a principal torus bundle over $F_{1}\times \cdots \times F_{r}$ with fiber of complex dimension $k$.
        \item If moreover $X$ is projective, it has a finite covering space $X_{0}$ biholomorphic to a direct product of Riemann surfaces with an Abelian variety. 
    \end{enumerate}
\end{thm}

As we shall see, the proof of this result relies on a classical argument going back to Siu's work~\cite{Siu-80}, as well as on some arguments taken from the work of Delzant and Py~\cite{DelPy-19}. Note that when $k=0$, the first item above says that $X$ is biholomorphic to the direct product of the $F_{j}$, and we recover the statement of \Cref{cor:strong-rigidity-of-direct-product}.

Before embarking into the proof, we remind the reader of the definition of the \emph{Albanese map} for a closed K\"ahler manifold. If $X$ is closed K\"ahler, one can associate a complex torus $A(X)$ of complex dimension $\frac{b_{1}(X)}{2}$ to it, along with a holomorphic map $a_X : X \to A(X)$ that is well-defined up to translation and such that the induced map 
\[ (a_{X})_{\ast} : H_{1}(X,\mathbb{Z})/\mathrm{Torsion}\to H_{1}(A(X),\mathbb{Z}) \]
is an isomorphism. We refer the reader to~\cite[\S 12.1.3]{voisin} for an exposition of this construction. When $X$ is a Riemann surface, $A(X)$ is nothing but the Jacobian of $X$. 

\begin{proof}
    Let $X$ be as in the statement of \Cref{thm:aspherical-case}. According to \Cref{thm:rigid-products-tori}, there is an isomorphism
    \[
        \psi : \pi_{1}(X) \longrightarrow \mathbb{Z}^{2k} \times \pi_{1}(S_{1})\times \cdots \times \pi_{1}(S_{r})
    \]
    and by a result by Catanese \cite[Thm.\ 5.14]{Cat-08}, there exist closed Riemann surfaces $F_{1}, \ldots , F_{r}$ respectively homeomorphic to $S_{1}, \ldots , S_{r}$ such that the surjections $\pi_{1}(X)\to \pi_{1}(S_{i})$ defined by $\psi$ are induced by some holomorphic fibrations $p_{i} : X \to F_{i}$ ($1\le i \le r$).
    
    We now argue as in~\cite[\S 4.2]{DelPy-19}. Let $a : X \to A$ be the Albanese map of $X$ and $\alpha_{i} : F_{i}\to A_{i}$ be the Albanese map of $F_{i}$. Note that $\alpha_{i}$ is an embedding. By the universal property of $\alpha$, there exists a holomorphic map $\varphi_{i} : A \to A_{i}$ making the diagram
    \[
        \begin{tikzcd}
            X \arrow[r, "a"] \arrow[d, "p_{i}"] & A \arrow[d, "\varphi_{i}"] \\
            F_{i} \arrow[r, "\alpha_{i}"] & A_{i} \\
        \end{tikzcd}
    \]
    commutative for each $i$. Denote by $\varphi : A \to A_{1}\times \cdots \times A_{r}$ the product of the maps $\varphi_{i}$ and by $\alpha : F_{1}\times \cdots \times F_{r} \to A_{1}\times \cdots \times A_{r}$ the product of the maps $\alpha_{i}$. We can and do assume that $\varphi$ is a homomorphism.
    
    Consider the submanifold
    \[
        Y\coloneq\{y\in A \mid \varphi (y)\in {\rm Im}(\alpha)\}
    \]
    of $A$. The map $a$ can be written as $a=i\circ \Phi$ where $\Phi : X \to Y$ is holomorphic and $i : Y \to A$ is the inclusion. One proves as in \cite[Lem.\ 35]{DelPy-19} that $\Phi$ is a homotopy equivalence. Finally, exactly as in~\cite{DelPy-19}, we invoke \Cref{lemma:siu}, stated below, to conclude that $\Phi$ is a biholomorphism. This proves the first item of the theorem, noting that the projection $Y\to {\rm Im} (\alpha)$ is a principal torus bundle.
    
    To prove the second item, we invoke (again as in~\cite{DelPy-19}) Poincar\'e's reducibility theorem~\cite[VI.8]{debarre}. If $X$ is projective, then so is $A$ (see~\cite{voisin}). Poincar\'e's theorem then implies that there is a subtorus $B\subset A$ intersecting $\ker \varphi$ transversely, in the sense that for any intersection point $x\in B\cap \ker \varphi$, we have \[ T_x A=T_x B \oplus T_x (\ker \varphi ).\] Hence, there are finitely many such intersection points. Consider now the natural product map $j: \ker \varphi \times B \to A$ taking $(x,y)\in \ker \varphi \times B$ to $x+y$. The transversality hypothesis implies that it is a local diffeomorphism, hence a finite covering map, i.e., an isogeny. This implies that $B$ is isogenous to the product of the $A_{i}$. Thus after replacing $B$ by a finite covering space we can assume that
    \[
        B=B_{1}\times \cdots \times B_{r},
    \]
    where $B_{i}$ is a finite covering space of $A_{i}$. We thus have a natural covering map
    \[
        \pi : \ker \varphi \times B_1 \times \cdots \times B_r \longrightarrow A.
    \]
    The preimage of $Y$ under $\pi$ is a finite covering space of $Y$ (hence of $X$) having the desired product structure. 
\end{proof}

A proof of the following classical lemma can be found in~\cite{Siu-80} (see the proof of Theorem 8 there).

\begin{lemma}\label{lemma:siu}
    Let $f : X_{1}\to X_{2}$ be a holomorphic map between two closed K\"ahler manifolds of the same dimension $n$. Assume that $f$ has degree $1$ and that the induced map $f_{\ast} : H_{2n-2}(X_{1},\mathbb{R})\to H_{2n-2}(X_{2},\mathbb{R})$ is injective. Then $f$ is a holomorphic diffeomorphism. This holds in particular if $f$ is a homotopy equivalence.   
\end{lemma}

\begin{remark}
    In (real) dimension greater or equal to $6$, one could appeal to the results of Farrell--Jones~\cite{FarJon-93} to prove that a manifold $X$ as in \Cref{thm:aspherical-case} must be homeomorphic to the direct product of the $S_{i}$'s with a torus. However, our proof based on K\"ahler geometry yields a stronger result.    
\end{remark}
    
\section{Profinite invariance of the BNS invariant}\label{sec.bns}

The goal of this section is to prove that the BNS invariant $\Sigma^1(G)$ of a K\"ahler group $G$ is a profinite invariant in a suitable sense, thereby proving \Cref{thmx.BNS.profinite}.  

We first recall some notions related to Alexander polynomials and present the $\TAP_1(R)$ property for a ring $R$ (\Cref{sec:alpolandtap}), which was introduced by Hughes and Kielak~\cite{HughesKielak2022}. The definition of $\TAP_1(R)$ will be recalled in \Cref{TAP.def}, but a more detailed introduction can be found in~\cite{HughesKielak2022}. We then establish in \Cref{sec:tapkahler} that K\"ahler groups have the $\TAP_1(R)$ property for every Noetherian unique factorization domain (UFD) $R$. Note that the corresponding result for $3$-manifold groups was proven in a series of papers by Friedl and Vidussi \cite{FriedlVidussi2008,FriedlVidussi2011annals,FriedlVidussi2013}. Finally, \Cref{thmx.BNS.profinite} is established in \Cref{sec:proofprofiniterigiditybns}.

\subsection{Preliminaries and the TAP property}\label{sec:alpolandtap}

The following definitions are taken from Friedl and Vidussi's survey \cite{FriedlVidussi2011survey} but restated in the language of group homology. Throughout \Cref{sec.bns}, $R$ is a unique factorisation domain (UFD), which is moreover assumed to be Noetherian. For a finitely generated $R$-module $M$, the fact that $R$ is Noetherian implies that we can find a presentation
\[
    \begin{tikzcd} 
        R^r\arrow[r,"P"] & R^s \arrow[r] & M \arrow[r] & 0
    \end{tikzcd}
\]
where $P$ is an $r\times s$ matrix with coefficients in $R$.

\begin{defn}
    The \emph{first elementary ideal} $E(M)$ of $M$ is the ideal of $R$ generated by all maximal (i.e., $s\times s$) minors of the matrix $P$ if $0<s\leq r$, is defined to be $\{0\}$ if $s>r$, and is $R$ if $s=0$.  
\end{defn}

Note that $E(M)$ is independent of the presentation. See~\cite{CrowellFox1963} or \cite[Lem.\ 4.4]{Turaev2001} for a proof of this fact. Since $R$ is a UFD, there exists a unique smallest principal ideal of $R$ that contains $E(M)$. A generator of this principal ideal is defined to be an \emph{order} of $M$, and an order is well defined up to muliplication by a unit in $R$~\cite[p.\ 18]{Turaev2001}. In what follows, if $A$ is an arbitrary ring, we denote by $A[t^{\pm 1}]$ the ring of Laurent polynomials in the indeterminate $t$ with coefficients in $A$. If $A$ is a Noetherian UFD, then so is $A[t^{\pm 1}]$.

Let $G$ be a finitely generated group and $\alpha : G\twoheadrightarrow Q$ be an epimorphism onto a finite group. Consider the ring of Laurent polynomials $RQ[t^{\pm 1}]$ with coefficients in the group ring $RQ$. For each element $\varphi\in H^1(G;\Z)$, we can define a $G$-action on $RQ[t^{\pm 1}]$ by the formula
\[
    g\cdot x = t^{\varphi(g)}\alpha(g)x
\]
for $g \in G, x \in  RQ[t^{\pm 1}]$. The first twisted (homological) Alexander module of $\varphi$ and $\alpha$ is defined to be $H_1(G;RQ[t^{\pm 1}])$. Observe that $H_1(G;RQ[t^{\pm 1}])$ has the structure of a left $R[t^{\pm 1}]$-module, and it is a finitely generated $R[t^{\pm 1}]$-module since $G$ is finitely generated.

\begin{defn}\label{def:ttaapp}
    The first twisted Alexander polynomial $\Delta_{1,R}^{\varphi,\alpha}(t)$ over $R$ with respect to $\varphi$ and $\alpha$ is defined to be the order of the twisted (homological) Alexander module of $\varphi$ and $\alpha$, treated as a left $R[t^{\pm 1}]$-module. When $\alpha$ is the trivial homomorphism, i.e., $Q=\{1\}$, we shall simply write $\Delta_{1,R}^{\varphi}$ instead of $\Delta_{1,R}^{\varphi, \alpha}$; this is the classical or untwisted Alexander polynomial.
\end{defn}

\begin{remark}
    The twisted Alexander polynomial, or rather merely its vanishing or nonvanishing, can be defined for more general rings; see~\cite[\S2.B]{HughesKielak2022}. However, we content ourselves here with the case where $R$ is a Noetherian UFD. 
\end{remark}

In the following, we shall say that a nonzero character $\varphi : G\to \Z$ is fibred if its kernel is finitely generated. Although the terminology algebraically fibred is more commonly used in the literature, we use the term fibred here since there is no risk of confusion with other notions. Being fibred is equivalent to the fact that $\varphi$ is contained in the intersection
\[
    \Sigma^1 (G)\cap - \Sigma^1(G).
\]
See~\cite{bieneustr} or~\cite[Ch.\ 11]{py-book}. Here, and as in the introduction, the BNS invariant $\Sigma^{1}(G)$ is viewed as a subset of $H^{1}(G,\R)-\{0\}$. Following~\cite{HughesKielak2022}, we say that $\varphi$ is \emph{semi-fibred} if it lies in $\Sigma^1(G)\cup -\Sigma^1(G)$. When $\Sigma^1(G)=-\Sigma^1(G)$, the two notions are equivalent. When $\varphi$ is semi-fibred, the Alexander polynomial $\Delta_{1,R}^{\phi,\alpha} (t)$ is nonzero for any surjective morphism to a finite group $\alpha : G \to Q$; see~\cite[Lem.\ 4.1]{FriedlVidussi2016}, or~\cite[Prop.\ 2.13]{HughesKielak2022}. Property $\TAP_1(R)$ deals with the converse of this implication.

\begin{defn}\label{TAP.def}
	  A finitely generated group $G$ is in the class $\TAP_1(R)$ if, for every nontrivial character $\varphi\in H^1(G;\Z)$, $\varphi$ is semi-fibred if and only if for every surjective homomorphism $\alpha: G\to Q$ to a finite group $Q$ its twisted Alexander polynomial $\Delta_{1,R}^{\varphi,\alpha}(t)$ is nonvanishing.
\end{defn}

Applications of this notion to some profinite rigidity problems are given in~\cite{HughesKielak2022}. Here we use it as a key step to establish \Cref{thmx.BNS.profinite}.  

\subsection{Large characters and the TAP property}\label{sec:tapkahler}

We continue to consider a finitely generated group $G$. We shall say that a character $\varphi : G\onto \Z$ is \emph{large} if there exists a finite-index subgroup $H$ of $G$, and a free noncyclic quotient $\psi : H\onto F$, such that $\varphi|_{H}$ factors through $\psi$. Note that if $\varphi$ is large, the group $H$ can be taken to be characteristic.

As in \Cref{sec:alpolandtap}, in what follows $R$ is a Noetherian UFD.

\begin{prop}\label{prop.largecharTAP}
    Let $G$ be a finitely generated group and let $\varphi\in H^1(G;\Z)$ be nontrivial and primitive. If $\varphi$ is large, then there exists a finite quotient $\alpha : G\onto Q$ such that $\Delta_{1,R}^{\varphi,\alpha}(t)=0$.
\end{prop}
\begin{proof}
    Since $\varphi$ is large, there exists a finite index characteristic subgroup $H$ of $G$ such that $\varphi|_H$ factors through a surjection $H\onto F_k$. Let $n$ be the divisibility of the character $\varphi|_H$, i.e., the index of the group $\varphi (H)$ in $\Z$. Applying a suitable automorphism of $F_k$ and then projecting onto a smaller free quotient, we can assume that $k=2$ and that $\varphi|_H$ equals the composition of the surjection $H\onto F_2=F(a,b)$ with the morphism $F(a,b)\to \Z$ sending $a$ to $n$ and $b$ to $0$. Let $Q=G/H$ and let $\alpha$ denote the quotient map $G\to Q$.  By \cite[Lem.\ 2.11]{HughesKielak2022} or \cite[Lem.\ 3.3]{FriedlVidussi2008AJM} we have 
    \begin{equation}\label{eq:apareequal}
        \Delta^{\varphi|_H}_{1,R}(t)=\Delta^{\varphi,\alpha}_{1,R}(t).
    \end{equation}
    (As explained in \Cref{def:ttaapp}, the left-hand side stands for the untwisted Alexander polynomial.) By~\cite[Lem.\ 2.9]{HughesKielak2022}, it suffices to prove that $H_1(H;R[t^{\pm 1}])$ is not a torsion $R[t^{\pm 1}]$-module. Indeed, this implies that the element in~\eqref{eq:apareequal} is zero and concludes the proof of the proposition.

    By Shapiro's Lemma there is an isomorphism
    \[
        H_1(H;R[t^{\pm1}])\cong H_1(\ker\varphi|_H;R)
    \]
    of $R[t^{\pm 1}]$-modules. We denote by $F_\infty$ the kernel of the homomorphism $F(a,b)\to \Z$ defined above. There is a commutative diagram
    \[
        \begin{tikzcd}
            H \arrow[r, two heads]  & F_2  \\
            \ker\varphi|_H \arrow[u] \arrow[r, two heads] & F_\infty \arrow[u]
        \end{tikzcd}
    \]
    and $H_1(\ker\varphi|_H;R)$ maps onto the module $H_{1}(F_\infty,R)$. Note that the conjugates
    \[
        a^{b^{j}}=b^{j}ab^{-j}
    \]
    for $j \in \ZZ$ form a basis of $F_{\infty}$. This yields an isomorphism
    \[
        H_1(F_{\infty},R)=\bigoplus_{\Z}R
    \]
    where the $R[t^{\pm 1}]$-module structure is obtained by defining $t$ to act by permutation of the factors. Thus $H_{1}(F_{\infty},R)$ is not a torsion module and thus $H_1(\ker\varphi|_H;R)$ also is not a torsion module. This completes the proof.
\end{proof}

\begin{corollary}\label{cor:largecharTAP}
    Let $G$ be a finitely generated group. If every nontrivial primitive character of $G$ that is not semifibred is large, then $G$ is in $\TAP_1(R)$.
\end{corollary}
\begin{proof}
    This is an immediate consequence of \Cref{prop.largecharTAP} and \cite[Prop.\ 2.13]{HughesKielak2022}, which says that the first twisted Alexander polynomials of semi-fibred characters do not vanish.
\end{proof}

To prove that K\"ahler groups are $\TAP_1(R)$, we will use a characterisation of large characters for K\"ahler groups in terms of factorisations through hyperbolic orbisurface groups. The following result, which is well-known, follows from the work of Napier--Ramachandran~\cite{NapRam-01} or from Delzant's description of the BNS-invariant for K\"ahler groups~\cite{Del-BNS}, that was recalled in the introduction. 

\begin{thm}\label{thm:Delzant}
    Let $G$ be a K\"ahler group and let $\chi\in H^1(G;\Z)$ be an integral character. The following are equivalent:
    \begin{enumerate}
        \item $\chi$ is large;
        \item $\chi$ is not fibred;
        \item $\chi$ is not semi-fibred;
        \item $\chi$ factors through one of the components of the universal homomorphism of $G$.
    \end{enumerate}
\end{thm}
\begin{proof} We have seen in the introduction that $\Sigma^1 (G)=-\Sigma^1 (G)$. This implies that the second and third points are equivalent. The first point implies the second since nontrivial infinite index normal subgroups of nonabelian finitely generated free groups are not finitely generated. The third point implies the fourth by~\cite{NapRam-01} or~\cite[Thm.\ 1.1]{Del-BNS}. Finally, it is a classical fact that any homomorphism from a hyperbolic surface group to $\Z$ factors through a surjection onto a nonabelian free group. Since hyperbolic orbisurface groups are virtually hyperbolic surface groups, this proves that the fourth point implies the first. 
\end{proof}

Combining~\Cref{cor:largecharTAP} and~\Cref{thm:Delzant}, we obtain:

\begin{prop}\label{prop.Kahler.TAP}
    If $G$ is a K\"ahler group, then $G$ is in $\TAP_1(R)$.
\end{prop}

\subsection{Proving~\Cref{thmx.BNS.profinite}}\label{sec:proofprofiniterigiditybns}

We start by recalling the statement we want to prove:

\setcounter{thmx}{3}
\begin{thmx}\label{thmx.BNS.profinite}
   Let $G$ and $H$ be K\"ahler groups such that $\widehat{G}\cong\widehat{H}$.  Then there is a linear isomorphism $L : H^{1}(G,\QQ) \to H^{1}(H,\QQ)$ such that
    \[
        L_{\RR}(\Sigma^1(G))=\Sigma^1(H).
    \]
\end{thmx}

In the above statement, recall that $L_{\RR}$ stands for the extension of $L$ to a map $H^1(G,\RR)\to H^{1}(H,\RR)$.

Before proving \Cref{thmx.BNS.profinite}, we need to recall several notions introduced by Liu~\cite{Liu23}. Let $A$ and $B$ be a pair of finitely generated free abelian groups and $\Phi : \widehat{A}\to\widehat{B}$ be a continuous homomorphism between their profinite completions. Any basis of $A$ as a $\Z$-module can be considered as a basis of $\widehat{A}$ as a $\widehat{\Z}$-module, and similarly for $B$. Therefore, if we fix two bases $v=(v_j)_{1\le j\le a}$ and $w=(w_t)_{1\le t\le b}$ of $A$ and $B$, respectively, the morphism $\Phi$ can be represented by a matrix $M(\Phi,v,w)$ of size $b\times a$, with coefficients in $\widehat{\Z}$. The following definition is due to Liu~\cite[\S 3]{Liu23}.

\begin{defn}\label{def:mc-phi}
The \emph{matrix coefficient module} 
$\mathrm{MC}(\Phi)$ of $\Phi$ is the $\Z$-submo\-dule of $\widehat{\Z}$ spanned by the coefficients of $M(\Phi , v,w)$. 
\end{defn}

One readily checks that this module does not depend on the choices of the bases $v$ and $w$. Alternatively, one can say that $\mathrm{MC}(\Phi)$ is the smallest $\Z$-submodule $L\subset \widehat{\Z}$ such that $\Phi(A)\subset B\otimes_{\Z}L$. By definition, $\mathrm{MC}(\Phi)$ is finitely generated. It is moreover free, since $\widehat{\Z}$ is torsion free. 

We now consider the homomorphism 
\[\begin{tikzcd}\mathrm{MC}(\Phi) \arrow[r]  & \Z/2\Z \end{tikzcd}\]
obtained by composing the inclusion of $\mathrm{MC}(\Phi)$ in $\widehat{\Z}$ with the canonical morphism $\widehat{\Z}\to \Z/2\Z$. Since $\mathrm{MC}(\Phi)$ is free, we can choose a lift ${\varepsilon : \mathrm{MC}(\Phi)\to \Z}$ of this morphism. Once $\varepsilon$ is fixed, we define the \emph{$\varepsilon$-specialization of $\Phi$} as the homomorphism $\Phi_{\varepsilon} : A \to B$ obtained by composing the two morphisms in the following diagram: 
    \[
        \begin{tikzcd}
            A \arrow[r, "\Phi"] & B\otimes_\Z \mathrm{MC}(\Phi) \arrow[r, "1\otimes\varepsilon"] & B\otimes_\Z \Z =B.
        \end{tikzcd}
    \]
We also consider the dual $\Phi^{\varepsilon}: H^1(B,\Z)\to H^1(A,\Z)$ of $\Phi_{\varepsilon}$, called the \emph{dual $\varepsilon$-specialization of $\Phi$}. Note that both $\Phi_{\varepsilon}$ and $\Phi^{\varepsilon}$ depend on $\varepsilon$. If $\Phi$ is an isomorphism, then $\Phi_{\varepsilon}$ and $\Phi^{\varepsilon}$ have finite cokernel; see \cite[Lem.\ 4.10]{HughesKielak2022}.

\begin{proof}[Proof of \Cref{thmx.BNS.profinite}]
    Let $G$ and $H$ be K\"ahler groups as in the statement of the theorem and let $A\coloneq H_{1}(G,\Z)/{\rm Torsion}$ and $B\coloneq H_{1}(H,\Z)/{\rm Torsion}$. Fix an isomorphism $\Psi : \widehat G\to\widehat H$ between the profinite completions of $G$ and $H$. Then $\Psi$ induces an isomorphism
    \[
        \widehat{H_{1}(G,\Z)} \longrightarrow \widehat{H_1(H,\Z)}
    \]
    that in turn induces an isomorphism
    \[
        \Phi : \widehat{A}\longrightarrow \widehat{B}.
    \]
    We can then apply the previous discussion to $\Phi$ and its coefficient module ${\mathrm{MC}(\Phi)\subset \widehat{\Z}}$. Fix a morphism $\varepsilon : \mathrm{MC}(\Phi)\to \Z$ as above and consider the dual $\varepsilon$-specialization $\Phi^{\varepsilon} : H^{1}(H,\Z)\to H^{1}(G,\Z)$. Since $\Phi^{\varepsilon}$ has finite cokernel, we obtain an isomorphism $\Phi^{\varepsilon}_{\R} :  H^1(H;\R)\to H^1(G;\R)$, induced by $\Phi^{\varepsilon}$.
    
    Before continuing the proof we recall that if $(\varrho_1 , \ldots , \varrho_r): G\to \Gamma_1 \times \cdots \times \Gamma_r$ is the universal homomorphism of $G$, and if
    \[
        V_i={\rm Im} (\varrho_{i}^{\ast} : H^{1}(\Gamma_{i},\R)\to H^{1}(G,\R))
    \]
    for $1\le i \le r$, then the BNS invariant of $G$ is the complement of the union
    \begin{equation}\label{eq:bnspre}
        \bigcup_{i=1}^{r}V_{i}\subset H^{1}(G,\R).
    \end{equation}
    Similarly the BNS invariant of $H$ is the complement of a finite union
    \begin{equation}\label{eq:bnspredeux}
        \bigcup_{j=1}^{r}W_{j}\subset H^{1}(H,\R),
    \end{equation}
    where the $W_j$ are obtained using the universal homomorphism of $H$. The number of subspaces in~\eqref{eq:bnspre} and~\eqref{eq:bnspredeux} are equal thanks to \Cref{thm:profinite-univ-hom}. We note that the $V_i$ and the $W_j$ are defined over $\mathbb{Q}$. We will use the following lemma:

    \begin{lemma}\label{lemma:preservefiberedness}
        Given $\varphi \in H^{1}(H,\Z)$, if $\Phi^{\varepsilon}(\varphi)$ is fibred then $\varphi$ is fibred.    
    \end{lemma}
    \begin{proof}
        This is a direct consequence of \cite[Thm.\ 4.12]{HughesKielak2022} with $n=1$.
    \end{proof}

    We now conclude the proof of \Cref{thmx.BNS.profinite} using the lemma. We want to prove that 
    \begin{equation}\label{eq:mapbnstobns}
        \Phi^{\varepsilon}_{\R}(\bigcup_{i=1}^{r}W_i)=\bigcup_{j=1}^{r}V_{j}.
    \end{equation}
    We will first prove that this equality holds for rational points. So let $a\in  \bigcup_{i=1}^{r}W_i$ be a rational point. According to \Cref{thm:Delzant} and \Cref{lemma:preservefiberedness}, $\Phi^{\varepsilon}(a)$ belongs to the union of the $V_j$. Since the $W_j$ are rationally defined, this implies that 
    \begin{equation}\label{eq:uneinclprte}
        \Phi^{\varepsilon}_{\R}(\bigcup_{j=1}^{r}W_{j})\subset \bigcup_{i=1}^{r}V_i.
    \end{equation}
    Note that each subspace $V_i$ (or $W_j$) has (positive) even real dimension, hence $V_{i}-\{0\}$ is connected. Since the $V_i$ have pairwise trivial intersection, the connected components of
    \[
        \left( \bigcup_{i=1}^{r}V_{i}\right) -\{0\}
    \]
    are the $V_{i}-\{0\}$. Since $\Phi^{\varepsilon}_{\R}(W_{j}-\{0\})$ is connected it must be contained in one of the $V_{i}-\{0\}$. Note that \Cref{thm:profinite-univ-hom} implies that the set of dimensions of the $V_i$ coincides with the set of dimension of the $W_{j}$. Up to permutation, we assume that
    \[
        {\rm dim}(W_{1})= \underset{1\le j\le r}{{\rm max}}\; {\rm dim}(W_{j})=\underset{1\le j\le r}{{\rm max}}\; {\rm dim}(V_{j}).
    \]
    If $i$ is such that $\Phi^{\varepsilon}_{\R}(W_{1}-\{0\})\subset V_{i}$, we then necessarily have $\Phi^{\varepsilon}_{\R}(W_{1})=V_{i}$. Up to reordering the $V_{t}$ we assume $i=1$. We thus have $\Phi^{\varepsilon}_{\R}(W_1)=V_1$ and
    \[
        \Phi^{\varepsilon}_{\R}(\bigcup_{j=2}^{r}W_j)\subset \bigcup_{i=2}^{r}V_{i}.
    \]
    Now, by choosing a subspace $W_{j_{0}}$ of maximal dimension among $W_{2}, \ldots , W_{r}$, we prove as above that $\Phi^{\varepsilon}_{\R}(W_{j_{0}})=V_i$ for some index $i\in \{2, \ldots , r\}$. Repeating this argument we see that the inclusion~\eqref{eq:uneinclprte} is promoted to equality in \eqref{eq:mapbnstobns} as desired.
\end{proof}

\section{Kodaira fibrations}\label{sec.quests}

Some important and well-studied K\"ahler groups not covered by the previous results in this paper are fundamental groups of Kodaira fibrations. Recall that a \emph{Kodaira fibration} $X$ is a complex surface endowed with a holomorphic submersion onto a Riemann surface that has connected fibers and is not isotrivial. It is known that the genus of both the base and the fiber must then be greater than $1$. See \cite{Cat-17} for an introduction to these complex surfaces. The fundamental group of a Kodaira fibration $X$ is called a \emph{Kodaira fibration group}. If $b$ and $f$ are respectively the genus of the base and fiber of the fibration, we say that $X$ and $\pi_1(X)$ have type $(f,b)$. Note that $X$ may fibre over a surface in more than one way, and so it may have multiple types.

The goal of this section is to raise the following structural question about normal subgroups of K\"ahler groups and explain its relation to the profinite rigidity problem for fundamental groups of Kodaira fibrations.

\begin{question}\label{Quest.NormalSubs.b1}
    Let $G$ be a K\"ahler group and $N\trianglelefteq G$ be a finitely generated normal subgroup. If $b_1^{(2)}(N)>0$, is $N$ virtually isomorphic to a surface group?
\end{question}

Note that if $N=G$, the answer is positive, according to an old result of Gromov (see~\cite{Gro-89} or~\cite[Th. 7.2]{py-book}). One could also consider the more restricted problem where $N$ is the kernel of the homomorphism to a hyperbolic orbisurface group induced by a fibration. A related question appears in~\cite[p.\ 3000]{nicolas}: it is unknown if a finitely generated nonabelian free group can appear as a normal subgroup of a K\"ahler group. The final result of this paper proves that a positive answer to \Cref{Quest.NormalSubs.b1} yields a weak form of profinite rigidity for Kodaira fibration groups.

\begin{prop}\label{Kodaira.conditional}
    Suppose that \Cref{Quest.NormalSubs.b1} has a positive answer and let $G$ be a Kodaira fibration group of type $(f,b)$. The following properties hold.
    \begin{enumerate}
        \item If $H\in\calg_\calk(G)$, then $H$ fits into an extension
         \[
        \begin{tikzcd}
            1 \arrow[r] & \pi_1 (\Sigma_f ) \arrow[r] & H \arrow[r] & \pi_1 (\Sigma_b ) \arrow[r] & 1, 
        \end{tikzcd}
    \]
    where $\Sigma_f$ and $\Sigma_b$ are closed oriented surfaces of genus $f$ and $b$ respectively. 
        \item If $X$ is a closed aspherical K\"ahler manifold such that $\pi_{1}(X)\in \calg_\calk(G)$, then $X$ is a Kodaira fibration. 
    \end{enumerate}
   \end{prop}
\begin{proof}
    First, we observe that the group $G$ is good in the sense of Serre by Lemma 7.2 and Theorem 7.3 in~\cite{Rei-15}. Then, by \cite[Prop.\ 7.10]{Rei-15}, there is a short exact sequence:
    \[
        \begin{tikzcd}
            1 \arrow[r] & \widehat{\pi_1 (\Sigma_f)} \arrow[r] & \widehat G \arrow[r] & \widehat{\pi_1 \Sigma_b)} \arrow[r] & 1
        \end{tikzcd}
    \]
    where $\Sigma_f$ and $\Sigma_b$ are closed oriented surfaces of genus $f$ and $b$ respectively. Now suppose $H\in\calg_\calk(G)$. Since $\pi_1(\Sigma_b)$ appears as a factor for the universal homomorphism of $G$, \Cref{thm:profinite-univ-hom} implies that there is a homomorphism $H\onto \pi_1(\Sigma_{b})$ with finitely generated kernel $N$. The closure of $N$ in $\widehat{H}$ is moreover isomorphic to $\widehat{\pi_1 (\Sigma_f)}$. \Cref{thm:l2-betti-fingen} implies that $b_1^{(2)}(N)>0$. If \Cref{Quest.NormalSubs.b1} has a positive answer, then $N$ is commensurable with a surface group, hence itself isomorphic to a surface group by the Nielsen realization problem.
    
    Note that $N$ is torsion free, since it is a subgroup of the torsion free group $\widehat{\pi_1 (\Sigma_f)}$ (see~\cite[Cor.\ 7.6]{Rei-15} for the latter point). It follows that $H$ is an extension
    \[
        \begin{tikzcd}
            1 \arrow[r] & \pi_1 (\Sigma_g) \arrow[r] & H \arrow[r] & \pi_1 (\Sigma_b) \arrow[r] & 1
        \end{tikzcd}
    \]
    where $\Sigma_g$ is a closed oriented surface of genus $g$. Then \cite[Lem.\ 7.2]{Rei-15} implies that $H$ is good, which in turn yields that the profinite completion of 
    \[
        N\cong \pi_1 ( \Sigma_g)
    \]
    is isomorphic to its closure in $\widehat{H}$, i.e., to $\widehat{\pi_1 (\Sigma_f)}$. This implies that $g=f$ and proves the first point of the proposition. To prove the second point of the proposition, if $X$ is a closed aspherical K\"ahler manifold such that $\pi_1 (X)\in \calg_\calk(G)$, then $\pi_{1}(X)$ is a surface-by-surface group by the first point, hence $X$ has complex dimension $2$. Applying~\cite[Prop.\ 1]{Kos-99} or~\cite{kapovich1998normal} implies that $X$ is a Kodaira fibration.
\end{proof}

\begin{remark}
    \Cref{Kodaira.conditional} suggests the following problem: is it true that a K\"ahler group that is also a surface-by-surface group can be realized as a Kodaira fibration group? Although this might be true, we did not try to establish this fact. See \cite{vidussi-note} for a related discussion.
\end{remark}

Despite the conditional result above, it is not clear whether or not one should expect fundamental groups of Kodaira fibrations to be profinitely rigid amongst residually finite K\"ahler groups. A result of Gonz\'alez-Diez \cite{GD-20} shows that for each extension of a nontrivial element $\sigma\in\mathrm{Gal}(\overline{\Q}/\Q)$ not conjugate to complex conjugation to an element of $\Aut(\CC)$, there exists a Kodaira fibration $X$ with a Galois conjugate $X^\sigma$ such that their universal covers are not isomorphic. It would be interesting to understand the fundamental groups of these examples in terms of whether or not Kodaira fibration groups can have arbitrarily large $\calk$-profinite genus.

\bibliographystyle{halpha}
\bibliography{refs.bib}

\end{document}